\documentclass{amsart}
\usepackage{math}
\usepackage{lineno,tikz,mathabx,mathrsfs,mathtools,hyperref}
\usetikzlibrary{cd,quotes,angles,decorations,arrows,calc,automata}

\newtheorem{obs}[thm]{Observation}

\usepackage{algorithm}
\usepackage[noend]{algpseudocode}
\newcommand*\Let[2]{\State #1 $\gets$ #2}

\renewcommand\tree{{\mathcal T}}

\begin{document}

\title[Generating Pólya trees and extending Cayley's formula]{An algorithm for uniform generation of unlabeled trees (Pólya trees), with an extension of Cayley's formula}
\author{Laurent Bartholdi}
\email{L.B.: laurent.bartholdi@gmail.com}
\address{L.B.: FR Mathematik+Informatik, Saarland University}
\author{Persi Diaconis}
\address{P.D.: Department of Mathematics and Statistics, Stanford University}
\date{November 26, 2024}
\thanks{L.B.\ gratefully acknowledges partial support from the ERC AdG grant 101097307.\\
P.D.\ gratefully acknowledges partial support from NSF grant 1954042.}
\begin{abstract}
  Pólya trees are rooted, unlabeled trees on $n$ vertices. This paper gives an efficient, new way to generate Pólya trees. This allows comparing typical unlabeled and labeled tree statistics and comparing asymptotic theorems with `reality'.

  Along the way, we give a product formula for the number of rooted labeled trees preserved by a given automorphism; this refines Cayley's formula.
\end{abstract}
\maketitle

\section{Introduction}
A classical theorem (Cayley's formula) says that there are $n^{n-2}$ \emph{labeled} trees on $\{1,\dots,n\}$, rooted at $1$:
\def\treeA#1#2#3#4{\node[fill,label=right:{#1}] (a) at (0,0) {};
  \node[label=right:{#2}] (b) at (0,-0.5) {};
  \node[label=right:{#3}] (c) at (0,-1) {};
  \node[label=right:{#4}] (d) at (0,-1.5) {};
  \draw (a) -- (b) -- (c) -- (d);}
\def\treeB#1#2#3#4{\node[fill,label=right:{#1}] (a) at (0,-0.1) {};
  \node[label=right:{#2}] (b) at (0,-0.7) {};
  \node[label=right:{#3}] (c) at (-0.3,-1.3) {};
  \node[label=right:{#4}] (d) at (0.3,-1.3) {};
  \draw (a) -- (b) -- (c) (b) -- (d);}
\def\treeC#1#2#3#4{\node[fill,label=right:{#1}] (a) at (0,-0.1) {};
  \node[label=right:{#2}] (b) at (-0.3,-0.7) {};
  \node[label=right:{#3}] (c) at (0.3,-0.7) {};
  \node[label=right:{#4}] (d) at (0.3,-1.3) {};
  \draw (a) -- (b) (a) -- (c) -- (d);}
\def\treeD#1#2#3#4{\node[fill,label=right:{#1}] (a) at (0,-0.3) {};
  \node[label=right:{#2}] (b) at (-0.5,-1.1) {};
  \node[label=right:{#3}] (c) at (0,-1.1) {};
  \node[label=right:{#4}] (d) at (0.5,-1.1) {};
  \draw (a) -- (b) (a) -- (c) (a) -- (d);}

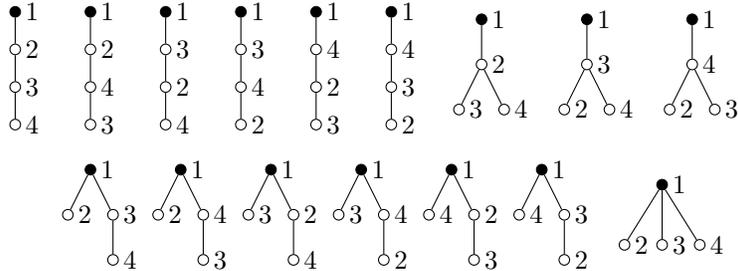
\begin{figure}[h!]
  \centerline{\begin{tikzpicture}[every node/.style={circle,inner sep=0pt,minimum size=4pt,draw}]
    \begin{scope}[xshift=0cm]\treeA1234\end{scope}
    \begin{scope}[xshift=1cm]\treeA1243\end{scope}
    \begin{scope}[xshift=2cm]\treeA1324\end{scope}
    \begin{scope}[xshift=3cm]\treeA1342\end{scope}
    \begin{scope}[xshift=4cm]\treeA1423\end{scope}
    \begin{scope}[xshift=5cm]\treeA1432\end{scope}
    \begin{scope}[xshift=6.2cm]\treeB1234\end{scope}
    \begin{scope}[xshift=7.6cm]\treeB1324\end{scope}
    \begin{scope}[xshift=9cm]\treeB1423\end{scope}
    \begin{scope}[yshift=-2cm,xshift=1cm]\treeC1234\end{scope}
    \begin{scope}[yshift=-2cm,xshift=2.2cm]\treeC1243\end{scope}
    \begin{scope}[yshift=-2cm,xshift=3.4cm]\treeC1324\end{scope}
    \begin{scope}[yshift=-2cm,xshift=4.6cm]\treeC1342\end{scope}
    \begin{scope}[yshift=-2cm,xshift=5.8cm]\treeC1423\end{scope}
    \begin{scope}[yshift=-2cm,xshift=7cm]\treeC1432\end{scope}
    \begin{scope}[yshift=-2cm,xshift=8.6cm]\treeD1234\end{scope}
  \end{tikzpicture}}
\caption{The $16$ labeled rooted trees at $1$, for $n=4$. Here and below the root vertex is indicated as solid, and non-root vertices are hollow.}
\end{figure}

\begin{figure}[h!]
  \centerline{\begin{tikzpicture}[every node/.style={circle,inner sep=0pt,minimum size=4pt,draw}]
    \begin{scope}[xshift=0cm]\treeA{}{}{}{}\end{scope}
    \begin{scope}[xshift=2cm]\treeB{}{}{}{}\end{scope}
    \begin{scope}[xshift=4cm]\treeC{}{}{}{}\end{scope}
    \begin{scope}[xshift=6cm]\treeD{}{}{}{}\end{scope}
  \end{tikzpicture}}
\caption{The $4$ unlabeled rooted trees, for $n=4$.}\label{fig:unlabeled}
\end{figure}
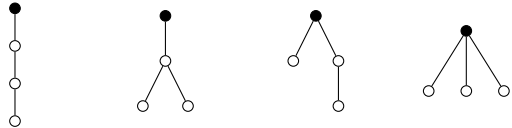

These are called \emph{Cayley trees} in his honour. In contrast, there is no formula for rooted, \emph{unlabeled} trees, called \emph{Pólya trees} to remember that Pólya~\cite{polya;chemische} (followed by Otter~\cite{otter;trees}) determined their asymptotics:
\begin{thm}\label{thm:otter}
  The number of rooted, unlabeled trees on $n$ vertices is asymptotic to
  \begin{equation}\label{eq:polyacount}
    \frac{b\sqrt{\rho}}{2\sqrt\pi}n^{-3/2}\rho^{-n}\big(1+\mathcal O(\tfrac1n)\big)
  \end{equation}
  with $b=2.6811266\dots$ and $\rho=0.338219\dots$.
\end{thm}

Even for small $n$, the number of labeled and unlabeled trees are quite different. The following numbers are from the Online Encyclopedia of Integer Sequences: \cite[A000081]{oeis} (Pólya) and~\cite[A000272]{oeis} (Cayley):
\begin{equation}\label{table:1}
  \begin{array}{c|cccccccccc}
    n & 1 & 2 & 3 & 4 & 5 & 6 & 7 & 8 & 9 & 10\\ \hline
    \text{Pólya} & 1 & 1 & 2 & 4 & 9 & 20 & 48 & 115 & 286 & 719\\
    \text{Cayley} & 1 & 1 & 3 & 16 & 125 & 1296 & 16807 & 262144 & 4782969 & 10^8
  \end{array}
\end{equation}

It is natural to suspect that properties of labeled and unlabeled trees will behave differently. There has been a good deal of elegant asymptotic theory for both classes of trees, reviewed in Section~\ref{sec:background}. The main result of this paper is to supplement the asmyptotics by providing a Monte Carlo algorithm to generate uniformly random Pólya trees (it is easy to generate random Cayley trees using Prüfer codes, as explained in~\S\ref{ss:cayleytrees}).

The algorithm gives a formula for the number of Cayley trees admitting a given automorphism. For integers $m,x,y$, consider $f(m,x,y)=(mx+y)^{m-1}y$. Let $\sigma$ be a permutation in $S_n$ fixing $1$, with $\lambda_d$ $d$-cycles. Set $\mu_d = \sum_{e|d} e\lambda_d$.
\begin{thm}[= Theorem~\ref{thm:count}]
  The number of labeled trees rooted at $1$ on $[n]$ that are $\sigma$-invariant is
  \[\lambda_1^{\lambda_1-2}\cdot\prod_{d\ge2,\lambda_d\neq0} f(\lambda_d,d,\mu_d).\]
\end{thm}
When $\sigma = \text{id}$, we have $\lambda_1 = n$, all trees are fixed and the formula boils down to $n^{n-2}$.

Background on trees and their asymptotic theory is collected in Section~\ref{sec:background}. We also review there the Burnside process, a general Monte Carlo Markov chain approach to generating random orbits for a finite group acting on a finite set --- this is the basis of our algorithm.

The algorithm itself is explained in Section~\ref{sec:algo}. Section~\ref{sec:compare} illustrates the algorithm. It compares the distribution of labeled and unlabeled trees, and compares asymptotic theory with simulations. It may be read directly for further motivation.

\subsection{Acknowledgments}
We thank David Aldous, Maciej Bendkowski, Sergey Dovgal, \'Eric Fusy, Slava Grigorchuk, Susan Holmes, Michael Howes, Ivan Mitrofanov, Evita Nestoridi, Jim Pitman, Carine Pivoteau, Nathan Tung and Stephan Wagner for their help with this project.

\section{Background}\label{sec:background}
This section reviews the extensive prior literature on trees (\S\ref{ss:trees}), Cayley trees (\S\ref{ss:cayleytrees}) and Pólya trees (\S\ref{ss:polyatrees}). These last two classes have global ``shape properties'' captured by Aldous' ``continuum random tree'' explained in~\S\ref{ss:crt}. The ``Burnside process'', which is the basis of our algorithm, is explained in~\S\ref{ss:burnside}.

\subsection{Trees}\label{ss:trees}
Because of their surprisingly wide applicability, there is a large combinatorial, enumerative and probabilistic literature on trees. Knuth~\cite{knuth;art-i} offers a succint review of the basics, see in particular~\cite[\S2.3.4.4]{knuth;art-i}. This is supplemented by the $40^+$-page development in~\cite[\S7.2.1.6]{knuth;art-4a}. The classical book by Moon~\cite{moon;counting-labelled-trees} develops probabilistic aspects which are brought much further in the readable account of Drmota~\cite{drmota;trees}. Any serious account of graph theory~\cite{diestel;graph-theory,bondy-murty;graph-theory,bollobas;modern-graph-theory} has a chapter on trees.

Of course, trees come in many flavours; phylogenetic trees~\cite{holmes;statistics-phylogenetic,felsenstein;inferring-phylogenetic}, real trees~\cite{evans;probability-real-trees}, A-B trees~\cite{odlyzko;new-tree-enumeration}, search trees~\cite{mahmoud;search-trees}, \dots. The present paper deals only with Cayley and Pólya trees, but we think the distinction between labeled and unlabeled, and our new algorithm, can be carried over to some of these classes.

Note that all the trees we consider are \emph{rooted}, namely come with a distinguished vertex. This is often dictated by applications (search trees, for example), but at any rate not a serious restriction: every tree has a central vertex or central edge, so it costs nothing to put the root there. Once a tree is rooted, it naturally admits an ordering of its edges towards the root, a notion of \emph{leaf}, and a recursive decomposition obtained by cutting the top branches off the root and viewing them as smaller rooted trees.

A wide variety of ``features'' are studied in the literature:
\begin{description}
\item[distances] the \emph{height} of a tree is the maximum distance from the root to a leaf, and its \emph{path length} is the sum of all distances to the root. Also of interest is ``pick a pair of vertices at random. What is the distribution of their distance?'' . This was Aldous' original motivation for his continuum random tree;
\item[growth] for $k=1,2\dots,\text{height}$, let $w_k$ be the number of vertices at distance $k$. The vector $(w_1,w_2,\dots,w_\text{height})$ is called the \emph{profile} and its maximum is called the \emph{width} of the tree;
\item[degrees] the maximum degree, the average degree, the empirical measure of degrees, and the number of leaves; we could list the degree of the root here, and also above as $w_1$.
\end{description}

For special applications, particularly phylogenetic trees, a variety of additional ``shape parameters'' are of interest. See~\cite{matsen} for a review of some, and the extensive book~\cite{fischer;tree-balance-indices}.

These features give rise to the question \emph{what does a typical tree ``look like''}?

\subsection{Cayley trees}\label{ss:cayleytrees}
There are \emph{many} different proofs that there are $n^{n-2}$ labeled trees rooted at $1$. See Lovasz~\cite{lovasz;combinatorial-problems} or \cite{addario-berry;foata-cayley}. One that we find algorithmically useful is via \emph{Prüfer codes}. This assigns to a tree $\tree$ a sequence of integers $[a_1,\dots,a_{n-2}]$, with $1\le a_i\le n$, giving a bijection between Cayley trees $\mathscr C_n$ and $[n]^{n-2}$. The mapping is easy to state. Given $\tree\in\mathscr C_n$, remove the leaf edge with the lowest leaf (the root never counts as a leaf) and let $a_1$ be the label on the other side of the edge. Continue with the remaining tree to generate $a_2,\dots,a_{n-2}$. For example,
\[\begin{tikzpicture}[every node/.style={circle,inner sep=0pt,minimum size=4pt,draw},baseline=-0.7cm]
    \node[label=right:1,fill] (a) at (0,0) {};
    \node[label=left:2] (b) at (-0.3,-0.7) {};
    \node[label=right:3] (c) at (0.3,-0.7) {};
    \node[label=left:4] (d) at (0,-1.4) {};
    \node[label=right:5] (e) at (0.6,-1.4) {};
    \draw (a) -- (b) (a) -- (c) -- (d) (c) -- (e);
  \end{tikzpicture}
  \leftrightarrow [1,3,3]
\]
since the leaves are removed in order $2,4,5$ and their neighbours are respectively $1,3,3$. It is similarly easy to reconstruct $\tree$ from the code.

By inspection, the degrees $d_i$ of vertices $1,\dots,n$ can be read from the code $[a_1,\dots,a_{n-2}]$ via $d_i=n_i+1$ if $i$ occurs $n_i$ times in the code. For the example above, the code $[1,3,3]$ has $d_1=2$, $d_2=d_4=d_5=1$, $d_3=3$. The number of trees with degree sequence $d_1,\dots,d_n$ is therefore
\[\binom{n-2}{d_1-1\quad\cdots\quad d_n-1}.\]

\noindent Translated as a probability statement, this becomes the following
\begin{prop}[Folklore]\label{prop:cayleydegrees}
  For $\tree\in\mathscr C_n$ chosen uniformly, the joint distribution of the degrees $(d_1,\dots,d_n)$ is exactly the same as the joint distribution of the box counts $(n_1,\dots,n_n)$ when $n-2$ balls are dropped randomly into $n$ boxes (with $d_i=n_i+1)$.\qed
\end{prop}

From this, standard facts about ``balls in boxes'' (multinomial allocation) yield theorems for trees. This is developed in~\cite{renyi;theory-trees,moon;counting-labelled} and elsewhere. We content ourselves with a corollary:
\begin{cor}
  For $\tree\in\mathscr C_n$ chosen uniformly, let $\ell(\tree)$ be the number of leaves. Then $\ell(\tree)$ has mean and variance
  \[\mathbb E_n(\ell(\tree))\sim\frac ne,\qquad\mathbb V_n(\ell(\tree))\sim n\frac1e(1-\tfrac1e)\]
  and, normalized by its mean and standard deviation, $\ell(\tree)$ has a limiting standard normal distribution.
\end{cor}
\begin{proof}
  This is just a translation of classical facts about $n-2$ balls dropped into $n$ boxes and the distribution of the empty cells, see~\cite{kolchin-chistyakov;combinatorial-problems}. For example, if
  \[X_i=\begin{cases}1&\text{ if box $i$ is empty},\\0&\text{ else}\end{cases},\qquad \ell(\tree)=\sum_{i=1}^n X_i;\]
  and $\mathbb E_n(X_i)=(1-\frac1n)^{n-2}=\frac1e(1+\mathcal O(\tfrac1n)$, $\mathbb E_n(X_i X_j)=(1-\frac2n)^{n-2}$ for $i\ne j$, and an easy variant of the central limit theorem yields the corollary.
\end{proof}
By symmetry, Proposition~\ref{prop:cayleydegrees} also applies to the degree of the root.

The maximum degree in a tree $\tree$ is one more than the maximal box count if $n-2$ balls are randomly dropped into $n$ boxes. By now, it is well known that the maximum of integer valued random variables tend \emph{not} to have limiting distributions. A nice paper of Carr, Goh and Schmutz~\cite{carr-goh-schmutz;degree} shows that the maximum is concentrated around $\lfloor\log n/\log\log n\rfloor$. This follows from a more refined result:
\begin{thm}\label{thm:Delta}
  Let $\tree\in\mathscr C_n$ be chosen uniformly, and let $\Delta(\tree)$ be the maximum degree. Then, for $k_n\sim\log n/\log\log n$,
  \[\mathbb P(\Delta(\tree)\le k_n) = \exp\Big(-\exp\big(\log n -k_n \log k_n +k_n -\frac12\log k_n -\log(e\sqrt{2\pi})+o(1)\big)\Big) + o(1).\]
\end{thm}
In this limit, $\Delta(\tree)$ has a continuous limiting distribution. This differs from the parallel result below (Theorem~\ref{thm:polyamaxdeg}) for Pólya trees.

As explained in Drmota~\cite[Example~1.9]{drmota;trees}, random Cayley trees are exactly distributed as a Galton-Watson process wth Poisson birth distribution, conditioned to have $n$ descendents at extinction. This implies that, as random metric spaces for the path metric, $\tree\in\mathscr C_n$ converges to Aldous' continuum random tree. Now, a host of limit theorems for ``global functions'' are available, see~\S\ref{ss:crt} below for details. We remark here that the number of leaves and the maximum degree are not suitably continuous so are not covered by the general theory; see the remark at the end of~\S\ref{ss:crt}.

\subsection{Pólya trees}\label{ss:polyatrees}
The symmetric group $S_n$ acts on labeled trees in $\mathscr C_n$ by permuting their labels, and the subgroup $S_{n-1}=\{\sigma:\sigma(1)=1\}$ acts on labeled trees rooted at $1$. The orbits are rooted, unlabeled trees. If $t_n$ is the number of Pólya trees on $n$ vertices and $t(x)=\sum_{n\ge1}t_n x^n$ is their generating function, Pólya~\cite{polya;chemische} proved the functional equation
\begin{equation}\label{eq:functional}
  t(x) = x \exp(t(x)+\frac12 t(x^2)+\frac13 t(x^3)+\cdots).
\end{equation}

Pólya~\cite{polya;chemische}, followed by Otter~\cite{otter;trees}, used this and delicate singularity analysis to prove the asymptotics~\eqref{eq:polyacount}.

The functional equation and further delicate analysis has been used by the analytic combinatorics community to derive remarkable limiting properties for various features of random Pólya trees. A survey is in~\cite[Section~3.6]{drmota;trees}. In particular, this includes the limiting distribution of the height, profile, and limiting degree distribution~\cite{panagiotou-sinha;vertices;graph}.

Later work includes a remarkable limit theorem for the size of the automorphism group of a random Pólya tree~\cite{olsson-wagner;distribution}: for a random rooted Pólya tree $\tree$,
\[\frac{\log\#\text{Aut}(\tree)-\mu n}{\sqrt n}\Longrightarrow\mathscr N(0,\sigma^2)\]
for $\mu=0.1373423\dots$ and $\sigma^2=0.1967696\dots$; thus ``random trees have exponentially many automorphisms''; this will be useful in section~\ref{sec:algo} when a random automorphism must be chosen.

Many of these developments use an algorithm for generating a random element of a combinatorial class in the presence of a functional equation such as~\eqref{eq:functional}. This is called the ``Boltzmann sampler''. It has been extensively developed by the Flajolet school, in the highly recommended original article~\cite{duchon-flajolet;boltzmann}; see also~\cite{flajolet-fusy-pivoteau} focusing on algorithms for generating unlabeled objects.

The Boltzmann sampler generates objects of random size $N$ and one must tune parameters to get the distribution of $N$ centered about a desired $n$, rejecting samples for which $N\neq n$. The community has introduced a host of techniques to aid this but at present writing much is art; see \S\ref{ss:boltzmann} for a rough speed comparison which is not favourable to the Boltzmann sampler. Indeed, for large $n$ our experiments show that thousands of samples must be discarded before one gets one of the exact degree needed. In our comparisons of data with limits of theoretical predictions, we definitely wanted a source of fixed-size trees.

An important recent development, after Drmota's book, is the finding that, like Cayley trees, a random Pólya tree converges, as a random metric space, to Aldous' continuum random tree. This long open problem was solved by~\cite{haas-miermont-2012} and later sharpened by~\cite{panagiotou-stufler;scaling-polya}. It was known that Pólya trees cannot be seen as Galton-Watson trees (which have Aldous' trees as limits). Panagiotou-Stufler show that a random Pólya tree has a ``spine'' that is a Galton-Watson tree, and is decorated with a collection of ``little trees'' that don't affect convergence. Their well-written paper uses the Boltzmann sampler as a theoretical base. Refinements delineating the size of the ``decorations'' (each about $\mathcal O(\log n)$) are in~\cite{gittenberger-jin-wallner;shape}. Note finally that, for random Pólya trees, \cite{robinson-schwenk;distribution} determine the mean number of leaves. It is asymptotic to $(0.438156\dots)n$ which is is greater than the value $n/e$ from Proposition~\ref{prop:cayleydegrees} above, so already this local feature distinguishes Cayley and Pólya trees.

Another difference between Cayley and Polya trees is seen from~\cite{goh-schmutz;unlabeled}. Theorem~\ref{thm:Delta} above shows that the maximum degree of a random Cayley tree has a continuous extreme value distribution. On the other hand, \cite{goh-schmutz;unlabeled} show:
\begin{thm}\label{thm:polyamaxdeg}
  For a random Pólya tree, the maximal degree is concentrated on at most 3 values.
\end{thm}

\subsection{The continuum random tree}\label{ss:crt}
As described in the Introduction, there are literally hundreds of flavours of trees, each with natural uniform distribution and associated limit theorems. Drmota's book~\cite{drmota;trees} is a splendid account, with careful proofs and full references. Another superb reference to the continuum random tree is~\cite{le-gall;random-trees-applications}.

Aldous~\cite{aldous;continuum-i,aldous;continuum-ii,aldous;continuum-iii} introduced a kind of Brownian motion on such spaces and shows how many classes of trees and measures have the same limit theory. Thus, a hard won theorem for some specific model holds for many others. He has written a splendid overview of the subject~\cite{aldous;continuum-ii}. His rough motivation follows: a tree has a unique shortest path between two vertices which makes the tree into a metric space. For many models, if two vertices are chosen at random their distance is of order $\sqrt n$. His continuum random tree models these metric spaces of trees in which the mean distance between vertices is of order $\sqrt n$.

The limit object, the ``continuum random tree'' (CRT), in a random, compact metric space. This uses Gromov-Hausdorff's idea that the space of all compact, Polish, metric spaces can itself be made into a metric space, call it $(\mathscr M,d_{GH})$. This is itself a Polish space and the well developed machine of weak convergence is in force.

To describe the limit object, recall that Brownian excursion on $[0,1]$ is just Brownian motion $B_t$ with $B_0=B_1=0$, conditioned to be non-negative. Call this process $E_t$, for $t\in[0,1]$. Use $E_t$ to define a (random) pseudo-metric on $[0,1]$ via
\[d_E(s,t)=E_s+E_t-2\min_{s\le u\le t}E_u.\]
The quotient of $[0,1]$ obtained by identifying points at distance $0$ from each other gives a random metric space $([0,1]/{\sim},d_E)$, the CRT.

Now consider a class of trees $\mathscr T_n$ on $n$ vertices. Each $\tree\in\mathscr T_n$ may be viewed as a metric space, and remains so if the metric is rescaled by dividing it by $\sqrt n$. Choosing $\tree$ uniformly in $\mathscr T_n$ gives a (random) metric space $(\tree,d_\tree/\sqrt n)$.

Aldous shows that for many classes of trees
\[(\tree,d_\tree/\sqrt n)\Longrightarrow ([0,1]/{\sim},d_E).\]
This class includes trees constructed from critical Galton Watson trees with progeny having finite variance (the limiting $E$ is rescaled by $\sigma$), conditioned to have $n$ total progeny.

Since Cayley trees are Galton Watson trees with $\text{Poisson}(1)$ births per generation, random Cayley trees have a CRT limit. Drmota~\cite[Theorems~4.8, 4.11]{drmota;trees} works out the distribution of height and shape of random Cayley trees from this point of view.

For quite a while, it was a mystery: are Pólya trees in the domain of the CRT? Drmota and Gittenberger~\cite[Theorem~1]{drmota-gittenberger;shape-unlabeled} showed that Pólya trees are \emph{not} Galton-Watson and indeed Aldous wrote ``the model ``all unordered unlabelled trees equally likely'' does \emph{not} fit into this set-up, and no simple probabilistic description is known''~\cite[page~29]{aldous;continuum-ii} (though later in the same paper he conjectures that Pólya trees have CRT limits!).

A breakthrough occured in work of Haas and Miermont~\cite{haas-miermont-2012} sharpened by Panagiotou and Stufler\cite{panagiotou-stufler;scaling-polya} and Gittenberger et al~\cite{gittenberger-jin-wallner;shape}. These authors show that a random Pólya tree has a large ``spine'' which is Galton-Watson. The spine is then ``decorated'' with small forests (of size $\Theta(\log n)$) which do not affect convergence to a rescaled CRT. The details are deep, beautiful mathematics, and the account of Panagiotou-Stufler is recommended reading.

From all of this, for ``global functions'', continuous in the Gromov-Hausdorff topology, one expects Cayley and Pólya trees to ``look the same''. The examples in Section~\ref{sec:compare} tell a more nuanced story. Of course, there are any number of ``local features'' where this limit theory doesn't hold and special purpose theory and simulation are the only game in town. We treat some of these in Section~\ref{sec:compare}. As a side note, we add that we have found it difficult to find a reasonable, useful description of the continuous functions on tree space.

\subsection{The Burnside process}\label{ss:burnside}
Counting and sampling problems often occur in the presence of a group action. We begin in a general abstract setting: let $\mathscr X$ be a finite set, and let $G$ be a finite group acting on the right on $\mathscr X$. This divides $\mathscr X$ into disjoint orbits
\[\mathscr X=\mathscr O_1\sqcup\cdots\sqcup\mathscr O_k.\]
Natural questions are:
\begin{itemize}
\item How many orbits are there?
\item What are typical orbit sizes?
\item Do the orbits have ``nice names'', that is, can they be labeled by a convenient coding?
\item How can one choose an orbit uniformly at random?
\end{itemize}
The Burnside process~\cite{goldberg-jerrum;burnside,diaconis;bose-einstein} gives a Markov chain Monte-Carlo approach to the last problem, and this permits progress on the first two problems.

The process runs as follows:
\begin{enumerate}
\item From $x\in\mathscr X$ choose $g\in G$ uniformly in the stabilizer $G_x=\{h\in G:x^h = x\}$;
\item From $g\in G$ choose $x\in\mathscr X$ uniformly in the fixed point set $\mathscr X_g=\{y\in\mathscr X:y^g=y\}$.
\end{enumerate}
In other words, we consider the bipartite graph with vertex set $\mathscr X\sqcup G$ and an edge between $x\in\mathscr X$ and $g\in G$ whenever $x^g=x$, and perform a simple random walk on this graph.

The two steps of the chain go from $x\in\mathscr X$ to $y\in\mathscr X$ with chance
\[K(x,y)=\frac1{\#G_x}\sum_{g\in G_x\cap G_y}\frac1{\#\mathscr X_g}.\]
It is easy to see that this gives an ergodic, reversible Markov chain on $\mathscr X$ with stationary distribution
\[\pi(x)=\frac{k^{-1}}{\#\mathscr O_x},\qquad k=\#\text{orbits}, x\in\mathscr O_x.\]
Thus, running the chain, and simply reporting the current orbit gives a symmetric Markov chain on $\{1,\dots,k\}$ with uniform stationary distribution.

This procedure has been applied when $\mathscr X=G$ with $x^g=g^{-1}x g$, so the orbits are conjugacy classes. When furthermore $G=S_n$, the classes are indexed by partitions of $n$ and the procedure gives a useful way to generate a random partition; see~\cite{diaconis-tung}.

The Burnside process is a close cousin of the celebrated Swedsen Wang algorithm used for Ising simulations; see~\cite{andersen-diaconis} for details. It is expected to converge rapidly. This can be proved in some cases~\cite{diaconis-zhong,rahmani;commuting-chain,paguyo;burnside-partitions} but careful analysis of the running time for complex problems, such as partitions or trees, are open research problems. We report some empirical results in Section~\ref{sec:compare}.

\section{The algorithm}\label{sec:algo}
The implementation of the Burnside process, on rooted trees, requires two basic operations:
\begin{enumerate}
\item Given a labeled rooted tree on $[n]$, select uniformly one of its automorphisms;
\item Given a permutation on $[n]$, select uniformly a tree fixed by that permutation.
\end{enumerate}
We shall describe algorithms that implement efficiently these two steps; but first review and extend a well-known encoding of labeled trees by their \emph{Prüfer codes}. This is needed for the algorithm, but also gives our refinement of Cayley's formula.

\subsection{Prüfer codes with automorphisms, and a refinement of Cayley's formula}\label{ss:counting}
Consider $\sigma\in S_n$, a permutation on $n$ points. We are interested in the set $\mathscr T_\sigma$ of labeled rooted trees $\tree$ on $[n]$ that are invariant under $\sigma$.

The crucial remark is that if $\tree$ is invariant under $\sigma$, then there is an associated tree $\tree/\sigma$ on the cycles of $\sigma$; in which the cycle containing $i$ is below the cycle containing $j$ in $\tree/\sigma$ if and only if $i$ is below $j$ in $\tree$. For example, consider the full binary tree and some quotients:
\[\begin{tikzpicture}[every node/.style={circle,inner sep=0pt,minimum size=4pt,draw}]
    \begin{scope}
      \node[fill,label=right:1] (a) at (0,0) {};
      \node[label=right:2,label=left:{$\tree={}$}] (b0) at (-0.6,-0.6) {};
      \node[label=right:3] (b1) at (0.6,-0.6) {};
      \node[label=right:4] (c00) at (-0.9,-1.2) {};
      \node[label=right:5] (c01) at (-0.3,-1.2) {};
      \node[label=right:6] (c10) at (0.3,-1.2) {};
      \node[label=right:7] (c11) at (0.9,-1.2) {};
      \draw (a) -- (b0) -- (c00) (a) -- (b1) -- (c10) (b0) -- (c01) (b1) -- (c11);
    \end{scope}
    \begin{scope}[xshift=3cm]
      \node[fill,label=right:1] (a) at (0,0) {};
      \node[label=right:2,label=left:{$\frac{\tree}{(4,5)}={}$}] (b0) at (-0.3,-0.6) {};
      \node[label=right:3] (b1) at (0.6,-0.6) {};
      \node[label=left:45] (c0) at (-0.3,-1.2) {};
      \node[label=right:6] (c10) at (0.3,-1.2) {};
      \node[label=right:7] (c11) at (0.9,-1.2) {};
      \draw (a) -- (b0) -- (c0) (a) -- (b1) -- (c10) (b1) -- (c11);
    \end{scope}      
    \begin{scope}[xshift=6.5cm]
      \node[fill,label=right:1] (a) at (0,0) {};
      \node[label=right:2,label=left:{$\frac{\tree}{(4,5)(6,7)}={}$}] (b0) at (-0.3,-0.6) {};
      \node[label=right:3] (b1) at (0.3,-0.6) {};
      \node[label=left:45] (c0) at (-0.3,-1.2) {};
      \node[label=right:67] (c1) at (0.3,-1.2) {};
      \draw (a) -- (b0) -- (c0) (a) -- (b1) -- (c1);
    \end{scope}      
    \begin{scope}[xshift=10cm]
      \node[fill,label=right:1] (a) at (0,0) {};
      \node[label=right:23,label=left:{$\frac{\tree}{(2,3)(4,6)(5,7)}={}$}] (b) at (0,-0.6) {};
      \node[label=left:45] (c0) at (-0.3,-1.2) {};
      \node[label=right:67] (c1) at (0.3,-1.2) {};
      \draw (a) -- (b) -- (c0) (b) -- (c1);
    \end{scope}      
    \begin{scope}[xshift=13.5cm]
      \node[fill,label=right:1] (a) at (0,0) {};
      \node[label=right:23,label=left:{$\frac{\tree}{(2,3)(4,6,5,7)}={}$}] (b) at (0,-0.6) {};
      \node[label=right:4567] (c) at (0,-1.2) {};
      \draw (a) -- (b) -- (c);
    \end{scope}    
  \end{tikzpicture}
\]
  
The number of trees $\tree$ giving rise to the same $\tree/\sigma$ is the product $\pi(\sigma)$ of the cycle lengths of $\sigma$; indeed each vertex in $\tree/\sigma$, labeled by a cycle of length $d$, corresponds to $d$ vertices in $\tree$ in some given cyclic order. There are $d$ ways of matching that cyclic order with the order of the vertex's parent, and all these choices may be made independently. In other words, we think of each such edge in $\tree/\sigma$ as a cable made of $d$ strands of $\tree$; they can be twisted in $d$ ways to recover a possible $\tree$.

On the other hand, not all trees on the set of cycles of $\sigma$ are legal; only those for which if $i$ is below $j$ (our trees grow downwards) then the cycle length at $j$ divides the cycle length at $i$. These preliminary remarks are in fact sufficient to count the number of trees in $\mathscr T_\sigma$.

We now get into specifics. The following is an easy variant of Prüfer codes:
\begin{lem}
  Let $X,Y$ be any sets, and consider $n\in\N$. There is a bijection between the following two objects:
  \begin{enumerate}
  \item Forests on $[n]$ with a choice of one root per tree, a label in $X$ on each edge, and a label in $Y$ on each root;
  \item Sequences in $([n]\times X\cup\{0\}\times Y)^n$ whose last term is in $\{0\}\times Y$.
  \end{enumerate}
\end{lem}
\begin{proof}
  Given a forest, construct a sequence as follows: select the lowest-numbered leaf in it, say $k$. Either $k$ has a parent $j$ in the forest, and a label $x\in X$ on the edge $\{j,k\}$, or $k$ is a root with a label $y\in Y$ on it. In the former case, the first term of the Prüfer sequence is $(j,x)$ while in the second case it is $(0,y)$. Repeat until all $n$ vertices have been removed, and note that the last vertex is perforce a root. We have produced a sequence as desired.

  Conversely, given a sequence, Algorithm~\ref{algo:prufer2tree} constructs a forest, and it is easy to verdify that both operations are mutual inverses.
\end{proof}

For example, consider the following forest:
\[\begin{tikzpicture}[vertex/.style={circle,inner sep=0pt,minimum size=4pt,draw},edge/.style={fill=white,inner sep=1pt,scale=1}]
    \begin{scope}
      \node[vertex,fill,label=right:8,label=above:$y_8$] (a) at (0,0) {};
      \node[vertex,label=right:1] (b) at (0,-1) {};
      \node[vertex,label=left:11] (c) at (-0.5,-2) {};
      \node[vertex,label=right:6] (d) at (0.5,-2) {};
      \draw (a) -- node[edge] {$x_1$} (b) -- node[edge] {$x_{11}$} (c) (b) -- node[edge] {$x_6$} (d);
    \end{scope}
    \begin{scope}[xshift=18mm]
      \node[vertex,fill,label=right:2,label=above:$y_2$] (a) at (0,0) {};
      \node[vertex,label=left:10] (b) at (-0.5,-1) {};
      \node[vertex,label=right:7] (c) at (0.5,-1) {};
      \node[vertex,label=right:5] (d) at (0.5,-2) {};
      \draw (a) -- node[edge] {$x_{10}$} (b) (a) -- node[edge] {$x_7$} (c) -- node[edge] {$x_5$} (d);
    \end{scope}
    \begin{scope}[xshift=36mm]
      \node[vertex,fill,label=right:9,label=above:$y_9$] (a) at (0,0) {};
      \node[vertex,label=right:3] (b) at (0,-1) {};
      \draw (a) -- node[edge] {$x_3$} (b);
    \end{scope}
    \begin{scope}[xshift=45mm]
      \node[vertex,fill,label=right:4,label=above:$y_4$] (a) at (0,0) {};
    \end{scope}
    
  \end{tikzpicture}
\]
The lowest-numbered leaf is $3$, so the sequence starts with $(9,x_3)$, and the edge $\{3,9\}$ is removed. The lowest-numbered leaf is then the root $4$, so the sequence continues with $(0,y_4)$; proceeding, the whole sequence is
\[(9,x_3)\,(0,y_4)\,(7,x_5)\,(1,x_6)\,(2,x_7)\,(0,y_9)\,(2,x_{10})\,(0,y_2)\,(1,x_{11})\,(8,x_1)\,(0,y_8).\]
Starting again from this sequence $\Sigma$, the counts of the numbers, increased by one, are
\[\begin{array}{c|ccccccccccc}
    n & 1 & 2 & 3 & 4 & 5 & 6 & 7 & 8 & 9 & 10 & 11\\ \hline
    \# & 3 & 3 & 1 & 1 & 1 & 1 & 2 & 2 & 2 & 1 & 1
  \end{array}.\]
For example, the numbers $1$ and $2$ each appear twice in $\Sigma$. Adding $1$ gives the $3,3$ that start the second line. The table is used, along with $\sigma$, to reconstruct as follows the original forest.

The first entry of $\Sigma$ is $(9,x_3)$ and the lowest-numbered $1$ in the table above is $3$ so $\{3,9\}$ is selected as an edge labeled $x_3$ while the count of $9$ is decreased to $1$. The second entry of $\Sigma$ is $(0,y_4)$ so $4$ is made a root with label $y_4$; proceeding, the original forest is reconstructed.

\begin{algorithm}
  \caption{A decorated tree from a generalized Prüfer sequence}\label{algo:prufer2tree}
  \begin{algorithmic}[1]
    \Require{$\Sigma$ is a sequence over $\{1,\dots,n\}\times X\cup\{0\}\times Y$, of length $n$, ending in $\{0\}\times Y$.}
    \Statex
    \Function{PrüferDecode}{$\Sigma$}
      \Let{$n$}{$\text{length}(\Sigma)$}
      \Let{$d$}{$(1,\dots,1)$}\Comment{local degrees at $1,\dots,n$}
      \For{$i\gets1,\dots,n$}
        \If{$\Sigma_i=(j,x),\;j>0$}
          \Let{$d_j$}{$d_j+1$}
        \EndIf
      \EndFor
      \Let{$g$}{$\text{empty\_graph}(n)$}
      \Let{$L$}{$\text{sorted\_heap}(\{j\mid d_j=1\})$}\Comment{leaves}
      \For{$i\gets1,\dots,n$}
        \Let{$k$}{$\text{pop\_min}(L)$}
        \Let{$d_k$}{$d_k-1$}
        \If{$\Sigma_i=(j,x),\;j>0$}
          \Let{$g$}{$g\cup\text{edge }(j,k)\text{ labeled }x$}
          \Let{$d_j$}{$d_j-1$}
          \If{$d_j=1$} \text{push $j$ on $L$}\EndIf
        \Else
          \State \text{$\Sigma_i=(0,y)$ and then mark $k$ as root labeled $xy$}
        \EndIf  
      \EndFor
      \State \Return{$g$}
    \EndFunction
  \end{algorithmic}
\end{algorithm}

\begin{cor}
  For any $m,x,y\in\N$, the number $f(m,x,y)$ of labeled forests on $[m]$ with a symbol in $[y]$ at every root and a symbol in $[x]$ at every non-root is given by
  \[f(m,x,y) = (mx+y)^{m-1}y.\]
\end{cor}
Note that one recovers in particular Cayley's count $\tfrac{\partial f}{\partial y}f(m,1,0)=m^{m-1}$ of rooted trees on $[m]$.
\begin{proof}
  There is a bijection between edges and non-root vertices, in which every edge is paired with its endpoint furthest from the root.
\end{proof}

\noindent The next theorem is our extension of Cayley's formula.
\begin{thm}\label{thm:count}
  Let $\sigma$ be a permutation in $S_n$ fixing $1$ and with $\lambda_d$ cycles of length $d$; so $\lambda\vdash n$. Furthermore set $\mu_d = \sum_{e | d, e<d} e\lambda_e$. Then the number of $1$-rooted labeled trees on $[n]$ that are $\sigma$-invariant is
  \[t_{n,\sigma}=\lambda_1^{\lambda_1-2}\cdot\prod_{d\ge2,\lambda_d\neq0} f(\lambda_d,d,\mu_d).\]
\end{thm}
\begin{proof}
  If $\sigma$ is the identity, this reduces to the classical Cayley enumeration formula. We proceed by induction on the maximal cycle length $d$ of $\sigma$. Let there be then $\lambda_d>0$ cycles of length $d$ in $\sigma$, and let $\sigma'$ be the permutation from which these cycles were removed.

  Every $\sigma$-invariant tree may then be produced as follows: choose first a $\sigma'$-invariant tree $T'$ on the complement in $[n]$ of the $d$-cycles. Choose then a forest $F$ on $[\lambda_d]$, with a label in $[d]$ on each non-root and a label in $[\mu_d]$ on each root. Duplicate each tree $d$ times in $F$, to support it on the union of the $d$-cycles, by choosing for each non-root vertex $\in[\lambda_d]$ a point in that cycle. Attach then each root to any $[\mu_d]$ vertices of $T'$ whose cycle length divides $d$. We have exhausted all the choices and decorations of $F$ to produce once every possible $\sigma$-invariant tree; so $t_{n,\sigma} = t_{n-d\lambda_d,\sigma'} f(d\lambda_d,d,\mu_d)$ and the claim follows by induction.
\end{proof}

Theorem~\ref{thm:count} bears some similarity to work done for phylogenetic trees~\cite{billey-konvalinka-matsen;trees,fusy;symmetries,matsen;tanglegrams}. Recall that a phylogenetic tree is a rooted, leaf labeled, binary tree with $n$ terminal nodes. The symmetric group acts on these trees and the orbits are `tree shapes'. The Burnside process applies here and gives a method for choosing a uniformly distributed tree shape. As part of their work, the above authors derived a nice formula for the number of phylogenetic trees fixed by a permutation $\sigma$. Call this number $A(n,\sigma)$. Supppose that $\sigma$ has cycles of length $\lambda_1,\lambda_2,\dots,\lambda_\ell$ in decreasing order. They give
\[A(n,\sigma) = \prod_{i=2}^\ell (2(\lambda_i +\cdots+\lambda_\ell)-1).\]
Note that taking $\sigma=\text{id}$ gives $\lambda_1=\dots=\lambda_n=1$ and recovers for $A(n,\text{id})$ the classical count $(2n-3)!!$ of phylogenetic trees. Neither formula is a consequence of the other, but they both admit a reasonably simple product form because the fixed subtrees under powers of $\sigma$ form an increasing exhaustion of the tree. Note also that in the case of these (binary) phylogenetic trees all cycle lengths are a power of $2$, so there are polynomially many conjugacy classes of $\sigma$ to consider, so the Burnside process can be implemented on a polynomial-sized stateset.

There is a parallel development in studying the action of the symmetric group on the set of functions from $[n]$ to itself, see~\cite{constantineau-labelle;fixed,meir-moon;random-mapping-patterns,mutafchiev;random-mapping-patterns}. Formulas similar to Theorem~\ref{thm:count} are obtained by these authors; for example, the number of functions $[n]\to[n]$ that are fixed by $\sigma\in S_n$ --- namely, commute with $\sigma$ --- is
\[\prod_{i=1}^n(\mu_d)^{\lambda_d}\]
in the notation of Theorem~\ref{thm:count}; that formula can be proven by similar arguments to ours, and fits nicely with Joyal's proof of Cayley's formula establishing a bijection between bi-rooted trees and self-maps of $[n]$, see~\cite[page~236]{aigner-ziegler;book}.

\subsection{From a labeled tree to a permutation fixing it}
The problem of computing the automorphism group of a graph is a long-standing one, and has been studied extensively, both theoretically~\cite{babai;group-graphs-algorithms} and practically~\cite{mckay-piperno,anders-schweitzer;dejavu}. If the automorphism group has been ``computed'' in a sufficiently convenient format, it is usually straightforward to produce a uniformly random element from the group.

For practical purposes, and for the applications and tests we have in mind (with $n\approx 10^6$), much more efficient procedures are necessary.

Edmonds~\cite[\S6-21]{busacker-saaty;finite-graphs} gives a procedure to test isomorphism of rooted trees, based on canonical labelings of vertices. Mathematically, it may be described as follows: label each vertex with the nested list consisting, in lexicographically increasing order, of the labels of its children. Thus leaves are labeled by the empty list, and the four non-isomorphic unlabeled rooted trees from Figure~\ref{fig:unlabeled} have root respectively labeled $(((())))$, $(((),()))$, $((),(()))$ and $((),(),())$. Two trees are isomorphic if and only if their root labels are equal.

This algorithm is analyzed in detail in~\cite[Theorem~3.3]{aho-hopcroft-ullman;design} where it is shown to run in linear time, as an application of bucket sorting. The point is that, in decreasing order of height, the nested lists can be replaced by integers: start by putting $0$ at each leaf; then proceed in decreasing order of height, first labeling vertices by lists of integers and then sorting the set of these lists and replacing them by their index in the sorted set. These integers are called \emph{i-numbers}, and two trees are isomorphic if and onlyf if they have, at each level, the same sets of lists of i-numbers. Again for the trees from Figure~\ref{fig:unlabeled}, the i-numbers are

\centerline{\begin{tikzpicture}[every node/.style={circle,inner sep=0pt,minimum size=4pt,draw}]
    \begin{scope}[xshift=0cm]\treeA{1=(1)}{1=(1)}{1=(0)}{0}\end{scope}
    \begin{scope}[xshift=2.5cm]\treeB{1=(1)}{1=(0,0)}{0}{0}\end{scope}
    \begin{scope}[xshift=5cm]\treeC{1=(0,1)}{0}{1=(0)}{0}\end{scope}
    \begin{scope}[xshift=7.5cm]\treeD{1=(0,0,0)}{0}{0}{0}\end{scope}
  \end{tikzpicture}}

Colbourn and Booth show that a small modification of the algorithm gives the automorphism group of the tree, also in linear time. (Here some care is necessary: the automorphism group is given by a collection of generators; it is crucial to show that there are at most $n-1$ generators, and that each of them can be given in a very compact format, as a permutation for which only the moved points are stored.) We content ourselves with an intermediate notion: there is a linear-time algorithm that computes the \emph{automorphism partition}, namely the partition of the vertex set $[n]$ into orbits of the symmetric group. It is based on the following
\begin{lem}[Colbourn and Booth~{\cite[Lemma~2.1]{colbourn-booth;trees}}]
  Two vertices $v,w$ of a tree are in the same orbit of its automorphism group if and only if the list of i-numbers from the root to $v$ coincides with the list of i-numbers from the root to $w$.
\end{lem}

Once the orbit partition of a tree's vertex set is known, there is a linear-time algorithm that produces a uniformly random permutation fixing the tree:
\begin{algorithm}
  \caption{A uniformly random permutation of $[n]$ preserving a tree on $[n]$}\label{algo:tree2perm}
  \begin{algorithmic}[1]
    \Require{$\tree$ is a tree on $[n]$ rooted at $1$, and $\sim$ is its automorphism partition.}
    \Statex
    \Function{Recurse}{$f,u$}\Comment{fill in $f$ on subtree at $u$}
        \For{$v\gets\text{descendants}(u)$}
          \Let{$N$}{$\{w\in\text{descendants}(f[u]) : w\sim v\text{ and }w\notin f\}$}
          \Let{$w$}{$\textsc{UniformRandom}(N)$}
          \Let{$f[v]$}{$w$}
          \State $\textsc{Recurse}(f,v)$
        \EndFor
    \EndFunction

    \Statex
    \Function{UniformPermutation}{$\tree$}
      \Let{$f$}{$(1,0,\dots,0)$}\Comment{the permutation, to be filled in}
      \State $\textsc{Recurse}(f,1)$
      \State \Return{$n\mapsto f[n]$ as a permutation}
    \EndFunction      
  \end{algorithmic}
\end{algorithm}

The validity of Algorithm~\ref{algo:tree2perm} comes from the simple description of the automorphism group of a rooted tree: consider a rooted tree $\tree$ with subtrees $\tree_1,\dots,\tree_d$ immediately attached to the root, and let $\lambda$ denote the partition of $[d]$ in which $i,j$ are in the same part precisely when $\tree_i,\tree_j$ are isomorphic. For each class $C\in\lambda$, choose an element $i(C)$ in it. Then
\[\text{Aut}(\tree)=\bigg(\prod_{i\in[d]}\text{Aut}(\tree_i)\bigg)\rtimes\bigg(\prod_{C\in\lambda}S_C\bigg);\]
namely, trees in each class $C$ may be interchanged arbitrarily, and may also be acted upon by their own automorphism group. Thus a uniformly random element of $\text{Aut}(\tree)$ may be obtained by selecting uniformly elements of the symmetric groups $S_C$ above, and then recursively selecting uniformly random elements of $\text{Aut}(\tree_{i(C)})$. The resulting permutation of $[n]$ is obtained by composing these operations, as done in lines~3--5 of Algorithm~\ref{algo:tree2perm} in which the image of a descendant is chosen so as to be the descendant of the already-chosen image, and has not yet been allocated in the image of the permutation.

In practice, we only implemented Algorithm~\ref{algo:tree2perm}, and did not implement the i-numbering algorithm, since there are very fast packages that compute the automorphism partition in close to linear time. The most efficient one we found is available at \url{https://automorphisms.org}, based on~\cite{anders-schweitzer-stiess}.

Note that, when it turns to practical and not asymptotic considerations, most graph manipulation libraries are ill-suited for large trees because of memory allocation overhead. The best data structure, to store a forest on $[n]$, consists of three integer lists of size $n$, pointing to the parent ($0$ for roots), the firstborn ($0$ for leaves) and the next sibling ($0$ for benjamins) of each vertex.

\subsection{From a permutation to a tree fixed by it}
Consider a permutation $\sigma$ of $[n]$, and a tree $T$ that is invariant under $\sigma$. For every $d\in[n]$, the restriction of $T$ to the $d$-cycles of $\sigma$ is a forest $F_d$ in which each tree appears $d$ times isomorphically. Furthermore, these $d$-uples of trees are grafted on the subtree $T_{|d}$ of $T$ spanned by the forests $F_e$ for $e|d$. We may thus construct $T$ iteratively, by going through all $d$ in increasing order, and selecting uniformly a forest $F_d$. In fact, to take into account the symmetry, we proceed as follows. Let $C_d$ be the $d$-cycles of $\sigma$, let $n_d$ be the number of $d$-cycles, and let $v_d$ be the number of vertices of $T_{|d}$. We construct a forest $F'$ on $n_d$, with a decoration in $[d]$ on each edge, and a decoration in $[v_d]$ on each root. The edge decorations determine how an edge in $F'$ may give rise to $d$ edges in $F_d$, while the root decorations determine how the root in $F'$ may be attached to a vertex in $F_{|d}$; see Algorithm~\ref{algo:perm2tree}.

\begin{algorithm}
  \caption{A uniform tree on $[n]$ invariant under a permutation of $[n]$}\label{algo:perm2tree}
  \begin{algorithmic}[1]
    \Require{$\sigma$ is a permutation of $[n]$ fixing $1$.}
    \Statex
    \Function{InvariantTree}{$\sigma$}
    \Let{$S_1$}{fixed points of $\sigma$}
    \Let{$p$}{$(\text{rand}([\#S_1]):i=1,\dots,\#S_1-2)$}
    \Let{$T$}{$\textsc{PrüferDecode}(p)$, recoded from $[\#S_1]$ to $S_1$}

    \For{$d\gets2,\dots,n$}
      \Let{$V_d$}{$\{(s,j)\mid s\in S_e,e<d,e|d,j\in s\}$}\Comment{where $F_d$ can be grafted}
      \Let{$S_d$}{$d$-cycles of $\sigma$}\Comment{$S_d$ is a list of lists}
      \Let{$W$}{$\{(i,j)\mid 1\le i\le\#S_d,1\le j\le d\}\cup\{(0,j)\mid 1\le j\le\#V_d\}$}
      \Let{$p$}{$(\text{rand}(W):i=1,\dots,\#S_d-1;(0,\text{rand}([\#V_d])))$}
      \Let{$F'$}{$\textsc{PrüferDecode}(p)$}
      \Let{$F_d$}{$\text{emptygraph}(n)$}
      \For{$(i,j,x)\gets\text{edges}(F')$}\Comment{$x$ is an edge label}
        \For{$k\gets1,\dots,d$}
          \Let{$F_d$}{$F_d\cup(S_d[i][k],S_d[j][(k+x)\bmod d])$}
        \EndFor
      \EndFor
      \For{$(r,y)\gets\text{roots}(F')$}\Comment{$y$ is a root label}
        \Let{$(s,j)$}{$V_d[y]$}
        \For{$k\gets1,\dots,d$}
          \Let{$F_d$}{$F_d\cup(S_d[r][k],s[(k+j)\bmod \#s])$}
        \EndFor
      \EndFor
      \Let{$T$}{$T\cup F_d$}
    \EndFor
    \State \Return{$g$}
  \EndFunction
\end{algorithmic}
\end{algorithm}

Note that the total number of random choices Algorithm~\ref{algo:perm2tree} may make coincides with the total number of trees given by Theorem~\ref{thm:count}.

Here is an example to illustrate the algorithm. Consider the permutation $\sigma=(5,6)(7,8)(9,10)(11,12,13)(14,15,16,17,18,19)(20,21,22,23,24,25)$ in $S_{25}$. To construct an invariant tree, we first choose uniformly a tree on the fixed points of $\sigma$, say\\
\centerline{\tikz[every node/.style={circle,inner sep=0pt,minimum size=4pt,draw}]{\treeB1423}.}
Then we consider the three $2$-cycles, and choose a forest on $\{\{5,6\},\{7,8\},\{9,10\}\}$, with an element of $\mathbb Z/2\mathbb Z$ on each edge and an element of the fixed-point-set $[4]$ at each root; for example,\\
\centerline{\begin{tikzpicture}[vertex/.style={circle,inner sep=0pt,minimum size=4pt,draw}]
    \node[vertex,fill,label=right:{$\{5,6\}$},label=left:4] (a) at (0,0) {};
  \node[vertex,label=right:{$\{7,8\}$}] (b) at (0,-1) {};
  \node[vertex,fill,label=right:{$\{9,10\}$.},label=left:1] (c) at (3,0) {};
  \draw (a) -- node[fill=white,inner sep=2pt] {$1$} (b);
\end{tikzpicture}}
We graft the double of this forest onto the previous tree, attaching each root by an edge to the indicated vertex, and matching $\{5,6\}$ to $\{7,8\}$ by a shift of $1$, namely $5$ to $8$ and $6$ to $7$:
\centerline{\begin{tikzpicture}[vertex/.style={circle,inner sep=0pt,minimum size=4pt,draw}]
    \node[vertex,fill,label=right:1] (a) at (0.5,0) {};
    \node[vertex,label=right:4] (b) at (0,-0.6) {};
    \node[vertex,label=right:9] (b0) at (0.7,-0.6) {};
    \node[vertex,label=right:10] (b1) at (1.4,-0.6) {};
    \node[vertex,label=right:2] (c0) at (-1.2,-1.2) {};
    \node[vertex,label=right:3] (c1) at (-0.4,-1.2) {};
    \node[vertex,label=right:5] (c2) at (0.4,-1.2) {};
    \node[vertex,label=right:6] (c3) at (1.2,-1.2) {};
    \node[vertex,label=right:8] (d2) at (0.4,-1.8) {};
    \node[vertex,label=right:7] (d3) at (1.2,-1.8) {};
    \draw (a) -- (b) -- (c0) (b) -- (c1) (b) -- (c2) -- (d2) (b) -- (c3) -- (d3) (a) -- (b0) (a) -- (b1);
  \end{tikzpicture}}
For the $3$-cycle, there is a unique choice of forest, namely a single root, again with a label in $[4]$ at the root, say $2$; grafting the tripling of that vertex at $2$ gives the tree\\
\centerline{\begin{tikzpicture}[vertex/.style={circle,inner sep=0pt,minimum size=4pt,draw}]
    \node[vertex,fill,label=right:1] (a) at (0.5,0) {};
    \node[vertex,label=right:4] (b) at (0,-0.6) {};
    \node[vertex,label=right:9] (b0) at (0.7,-0.6) {};
    \node[vertex,label=right:10] (b1) at (1.4,-0.6) {};
    \node[vertex,label=right:2] (c0) at (-1.2,-1.2) {};
    \node[vertex,label=right:3] (c1) at (-0.4,-1.2) {};
    \node[vertex,label=right:5] (c2) at (0.4,-1.2) {};
    \node[vertex,label=right:6] (c3) at (1.2,-1.2) {};
    \node[vertex,label=right:8] (d2) at (0.4,-1.8) {};
    \node[vertex,label=right:7.] (d3) at (1.2,-1.8) {};
    \node[vertex,label=right:11] (e0) at (-2,-1.8) {};
    \node[vertex,label=right:12] (e1) at (-1.2,-1.8) {};
    \node[vertex,label=right:13] (e2) at (-0.4,-1.8) {};
    \draw (a) -- (b) -- (c0) (b) -- (c1) (b) -- (c2) -- (d2) (b) -- (c3) -- (d3) (a) -- (b0) (a) -- (b1) (c0) -- (e0) (c0) -- (e1) (c0) -- (e2);
  \end{tikzpicture}}
Finally there are two $6$-cycles, for which we must choose a forest on them (say the one-edge tree), a label for the root among the points of order dividing $6$, namely $[11]$ (say $10$) and an element in $\Z/6\Z$ on the edge (say $2$). Grafting six copies of that edge at $10$, and matching the $6$-cycles with a rotation of $2$, produces the tree
\begin{equation}\label{eq:tree25}
  \begin{tikzpicture}[vertex/.style={circle,inner sep=0pt,minimum size=4pt,draw}]
    \node[vertex,fill,label=above right:1] (a) at (0.5,0) {};
    \node[vertex,label=right:4] (b) at (0,-0.6) {};
    \node[vertex,label=below left:9] (b0) at (3.2,-0.6) {};
    \node[vertex,label=above right:10] (b1) at (4.8,-0.6) {};
    \node[vertex,label=right:2] (c0) at (-1.2,-1.2) {};
    \node[vertex,label=right:3] (c1) at (-0.4,-1.2) {};
    \node[vertex,label=right:5] (c2) at (0.4,-1.2) {};
    \node[vertex,label=right:6] (c3) at (1.2,-1.2) {};
    \node[vertex,label=right:8] (d2) at (0.4,-1.8) {};
    \node[vertex,label=right:7] (d3) at (1.2,-1.8) {};
    \node[vertex,label=right:11] (e0) at (-2,-1.8) {};
    \node[vertex,label=right:12] (e1) at (-1.2,-1.8) {};
    \node[vertex,label=right:13] (e2) at (-0.4,-1.8) {};
    \draw (a) -- (b) -- (c0) (b) -- (c1) (b) -- (c2) -- (d2) (b) -- (c3) -- (d3) (a) -- (b0) (a) -- (b1) (c0) -- (e0) (c0) -- (e1) (c0) -- (e2);
    \foreach\i in {0,1,2,3,4,5} {
      \pgfmathparse{int(14+\i)}
      \node[vertex,label=right:\pgfmathresult] (f\i) at (2+0.8*\i,-1.2) {};
      \pgfmathparse{int(22+mod(2+\i,6))}
      \node[vertex,label=right:{\pgfmathresult\ifnum\i=5.\fi}] (g\i) at (2+0.8*\i,-1.8) {};
      \pgfmathparse{int(mod(1+\i,2))}
      \draw (b\pgfmathresult) -- (f\i) -- (g\i);      
    }
  \end{tikzpicture}
\end{equation}

\subsection{$\sigma$-Prüfer sequences}
Here is a more extensive discussion of the encoding of trees with symmetries by a Prüfer sequence.

We consider a permutation $\sigma$ of $[n]$ fixing $1$. Say its number of $d$-cycles is $\lambda_d$ as above, so $\lambda$ is a partition of $n$. Let $\Lambda_d$ be the union of $\sigma$-cycles of length dividing $d$, namely those $i\in[n]$ with $\sigma^d(i)=i$, and let $\Lambda'_d$ be the subset of $\Lambda_d$ consisting of points of order strictly dividing $d$.

A \emph{$\sigma$-Prüfer sequence} is a sequence of numbers in $[n]$ of length $(\lambda_1-1)+\lambda_2+\dots+\lambda_n$, namely one less than the number of cycles of $\sigma$, subject to the following restriction. The sequence is grouped in blocks as in the sum above, so there is one block of numbers for each cycle length. The entries in the block of cycles of length $d$ have to belong to $\Lambda_d$, and the last entry in each such block has furthermore to belong to $\Lambda'_d$, with the convention that $\Lambda'_1=\{1\}$.
\begin{prop}
  There is a bijection between $\sigma$-Prüfer sequences and trees on $[n]$, rooted at $1$ and $\sigma$-invariant.\qed
\end{prop}
This is really just a reformulation of Theorem~\ref{thm:count}.

Let us illustrate this with three remarks. Firstly, if $\sigma=\text{id}$, then a $\sigma$-Prüfer sequence is really just a classical Prüfer sequence in $[n]^{n-2}$, with an extra $1$ appended.

Consider now as example the permutation $\sigma\in S_{25}$ from the previous section, and the corresponding tree from~\eqref{eq:tree25}. Let us re-encode it as a $\sigma$-Prüfer sequence. Because of the cycle counts of $\sigma$, this sequence will take the form
\[(*,*,1|*,*,*|*|*,*)\]
where the blocks of period $1,2,3,6$ have been separated by $|$. We first consider the fixed points of $\sigma$, and the corresponding subtree on $[4]$. Within that subtree, the smallest leaf is $2$, and its neighbour is $4$; the next smallest leaf is $3$ also with neighbour $4$; then the next smallest leaf is $4$ with neighbour $1$. Therefore the $\sigma$-Prüfer sequence starts as
\[(4,4,1|*,*,*|*|*,*).\]
We now move on to $2$-cycles and the corresponding subtree on $[10]$. The smallest leaf cycle is $(5,6)$ attached to $4$; the next-smallest is $(7,8)$ with minimum attached to $5$; the next-smallest $(9,10)$ attached to $1$; so the $\sigma$-Prüfer sequence continues as
\[(4,4,1|4,5,1|*|*,*).\]
In the subtree spanned by $\Lambda_3$, namely on $\{1,2,3,4,11,12,13\}$, there is a single leaf $3$-cycle $(11,12,13)$ attached to $2$, so the $\sigma$-Prüfer sequence continues as
\[(4,4,1|4,5,1|2|*,*).\]
Finally in the tree the smallest leaf $6$-cycle is $(24,25,26,27,22,23)$, with minimum attached to $18$; and after removing that $6$-cycle the smallest leaf $6$-cycle is $(14,15,16,17,18,19)$ with minimum attached to $10$. The $\sigma$-Prüfer sequence is therefore
\[(4,4,1|4,5,1|2|18,10).\]
It is clear that this process produces a bijection, essentially by following the steps in reverse.

The third remark is that there is some asymmetry in the handling of fixed points. This asymmetry vanishes if we count invariant forests rather than trees. For this, we consider an arbitrary permutation $\sigma\in S_n$, now set
\[\Lambda_d=\{0\}\cup\{i\in[n]:\sigma^d(i)=i\},\]
and let $\Lambda'_d\subseteq\Lambda_d$ consist of points of period strictly dividing $d$, with $0$ included. An \emph{extended $\sigma$-Prüfer sequence} a sequence of length $\lambda_1+\dots+\lambda_d$, such that the entries in the block of $d$-cycles belong to $\Lambda_d$, and the last entry in that block belongs to $\Lambda'_d$.
\begin{prop}
  There is a bijection between extended $\sigma$-Prüfer sequences and $\sigma$-invariant rooted forests on $[n]$.
\end{prop}
\begin{proof}
  Given a forest on $[n]$, extend it to a tree on $\{0\}\cup[n]$ by connecting all roots to $0$. This tree is rooted at $0$, and $\sigma'$-invariant for the permutation $\sigma'$ of $\{0\}\cup[n]$ that extends $\sigma$ by fixing $0$. Apply the previous proposition.
\end{proof}
(This is merely the generalization of the classical fact that Prüfer sequences in $[n+1]^{n-1}$ code forests on $[n]$).

\section{Comparison}\label{sec:compare}

The algorithms from Section~\ref{sec:algo} were implemented in the computer language \textsc{Julia}, and thanks to their essentially linear complexity were run to produce large numbers (in the millions) of random trees with large numbers of vertices (also in the millions). We have tested our code in various ways, checking in particular that, for small $n$, the proportion of Pólya trees produced matches the exact counts in~\ref{table:1}.

This section gathers some empirical data that can be extracted from these experiments. The first, and foremost, observation is that the Burnside process performs remarkably well. It is of course difficult to judge how close a sample is to the uniform distribution if the uniform distribution itself is not known; but we observe:

\begin{obs}
  Starting with the (obviously not random) tree of height $1$, given by the Prüfer code $[1,\dots,1]$, approximately $20$ steps of the Burnside process are sufficient to produce a uniformly random tree; i.e.\ a tree for which height, width, number of leaves, etc. have converged to their expected value.

  This has been checked experimentally for a number of vertices up to $n\approx 10^7$; and could be pushed further with more optimizations.
\end{obs}

We shall consider in turn the ``features'' (distances, growth, degrees) from~\S\ref{ss:trees}, comparing experimental data with those predicted using the constants $b,\rho$ appearing in Theorem~\ref{thm:otter}, and in particular with $\sigma=b\sqrt{\rho/2}$, the variance of the associated Galton-Watson process. We record our route to computing them at the end of this section; for reference, they are
\[\rho\approx 0.3383218568992077,\qquad b\approx 2.681128147267112,\qquad \sigma\approx 1.1027259685996555.\]

There are actually three kinds of comparisons that we have in mind:
\begin{enumerate}
\item compare our sampler (the ``Burnside sampler'') to others, and in particular the ``Boltzmann sampler'' mentioned in~\S\ref{ss:polyatrees};
\item compare features of Cayley and Pólya trees;
\item compare experimental data to theoretical predictions.
\end{enumerate}

\subsection{Comparing samplers}\label{ss:boltzmann}
We have compared our code to freely available implementations of the Boltzmann sampler. One of them, ``usain-boltz'', is distributed as the Python package \texttt{usainboltz}. It can generate structures following a grammar, and attempts to tune parameters so as to produce objects with the correct cardinality. For example, producing a Pólya tree is in principle obtained by the grammar rules
\begin{verbatim}
  z = Atom()
  B = RuleName("B")
  grammar = Grammar({B: z * MSet(B)})
  generator = Generator(grammar, B, singular=True)
\end{verbatim}
expressing $B$ as an atom $z$ followed by an unordered collection of $B$'s. The command \verb+generator.sample((1000,1000))+ produces a random tree with $1000$ vertices in roughly $2$ seconds. Another freely-available sampler was written by Carine Pivoteau, see\\
\centerline{\url{https://github.com/CarinePivoteau/Alea2023Notebooks}}
It also consists of easy-to-interface Python code, and can also produce a Pólya tree with $1000$ vertices in roughly $2$ seconds. A rough comparison (qq plot of tree depth of her trees compared to ours, see the next subsection) indicated that all three samplers are either correct, or at least suffer from similar bugs:
\[\includegraphics[width=7cm]{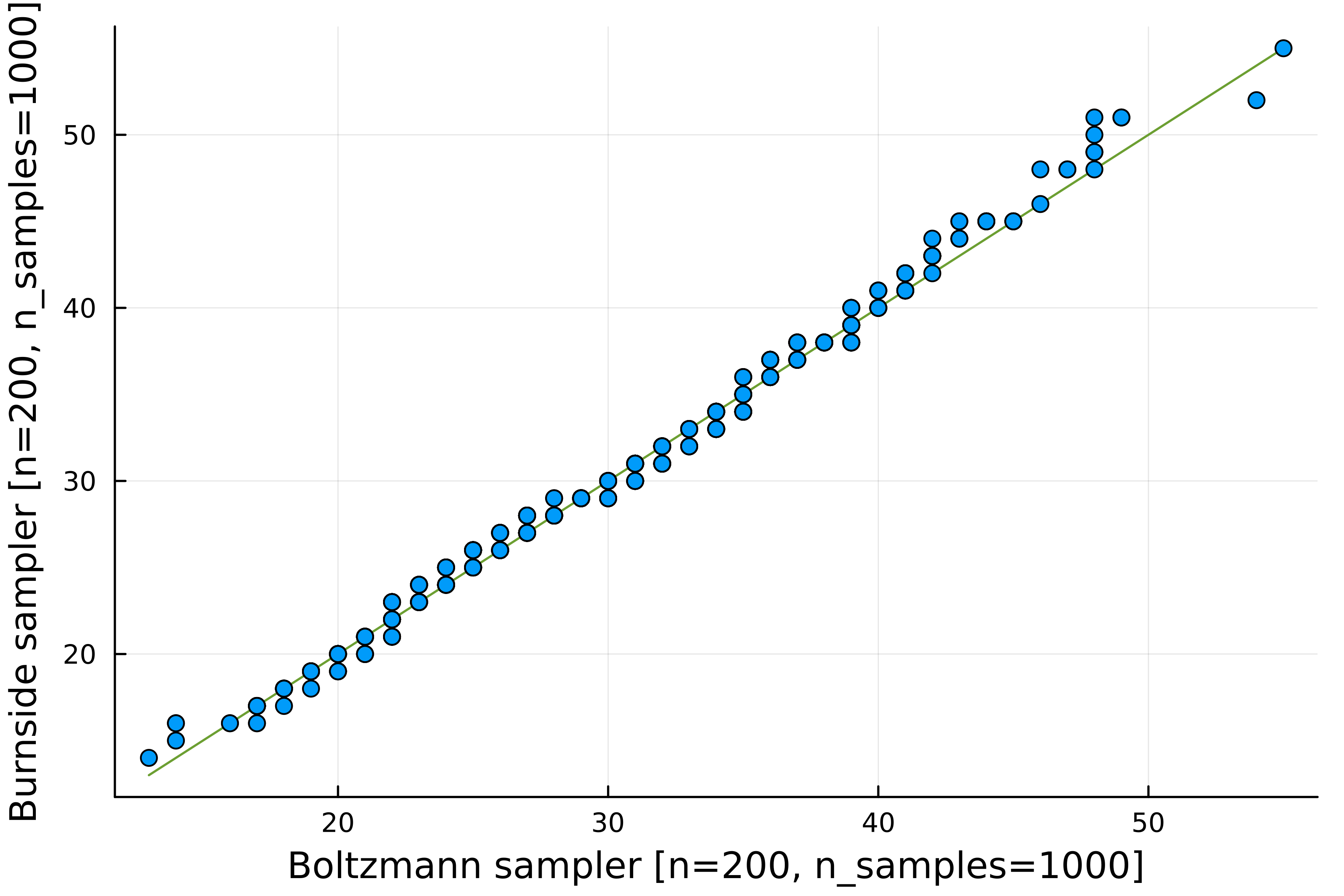}\]

We stress again the fact that the Boltzmann sampler doesn't return a random Pólya tree of exact degree $n$, but rather a random tree of degree $N$ where $N$ is random with mean $n$. For large $n$, our experience shows it can take thousands of sample to get one of the exact degree needed. Of course, fixed degree is crucial for tasks such as comparing data to limit approximations.

\subsection{Distances}
One well-understood statistic is the height $H_n$ of a random Pólya tree. A Galton-Watson trees whose offspring distribution has variance $\sigma$ has, according to~\cite[Theorem 4.8]{drmota;trees}, normalized height $H_n/\sqrt n$ converging in distribution to $2/\sigma\max_{0\le t\le1}E_t(t)$, where $E_t$ is the Brownian excursion, see~\S\ref{ss:crt}. There is an analytic expression for this maximum, sometimes called the \emph{Kolmogorov distribution}:
\[M(x)=1-\prod_{k\ge1}2(4k^2x^2-1)\exp(-2k^2x^2).\]
Since height is a continuous parameter, we deduce, following~\cite[Theorem 4.59]{drmota;trees},
\[H(\tree_n)/\sqrt n\approx\frac{2\sqrt2}{b\sqrt\rho}\max_{t\in[0,1]}E_t.\] 

We computed the heights of 1\,000\,000 random trees with 1000 vertices to the analytic estimate, and fitted a function of the form $dM(\mu_e+\sigma_e x)/d(\mu_e+\sigma_e x)$ for optimal $\mu_e,\sigma_e$ to the data, obtaining an empirical estimate $\sigma_e$ of $\sigma$. The result, and its qq-plot compared to the analytic Kolmogorov distribution is
\[\includegraphics[width=7cm]{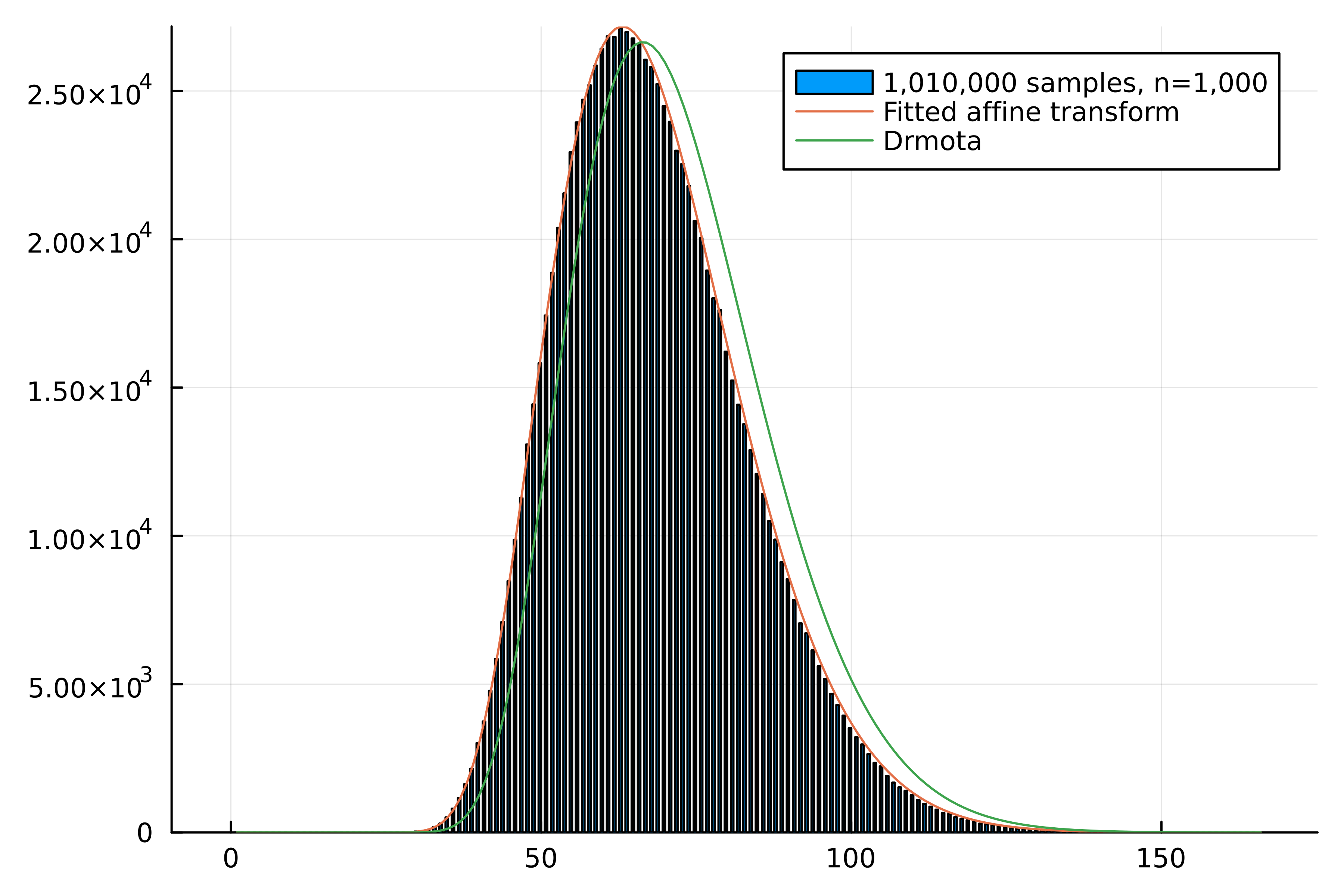}\qquad\includegraphics[width=7cm]{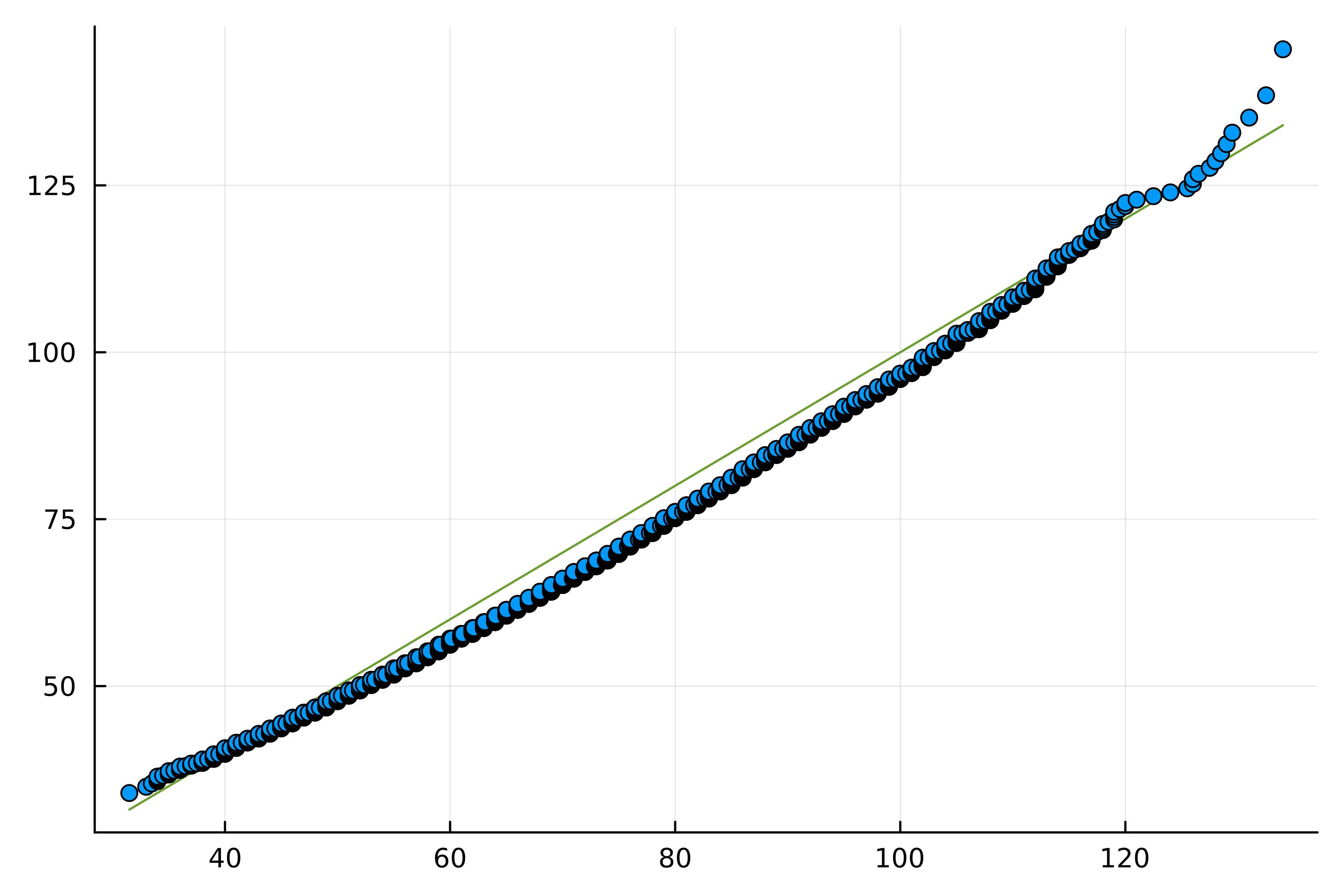}\]
One sees that the optimal fit, on the left, is extremely good, but quite far from the asymptotics; while the qq-plot shows significant differences. In particular, the optimal $\sigma_e$ turns out to be approximatively $1.04$, and this value persists for $n=10\,000$ and even $n=100\,000$; so if there is convergence of $\sigma_e$ to $1.10$ then it certainly is very slow.

The comparison with Cayley trees also gives quite clear results; the distributions are similar, but with a different scale. Here is the qq plot of depth of Cayley trees compared to Pólya trees:
\[\includegraphics[width=7cm]{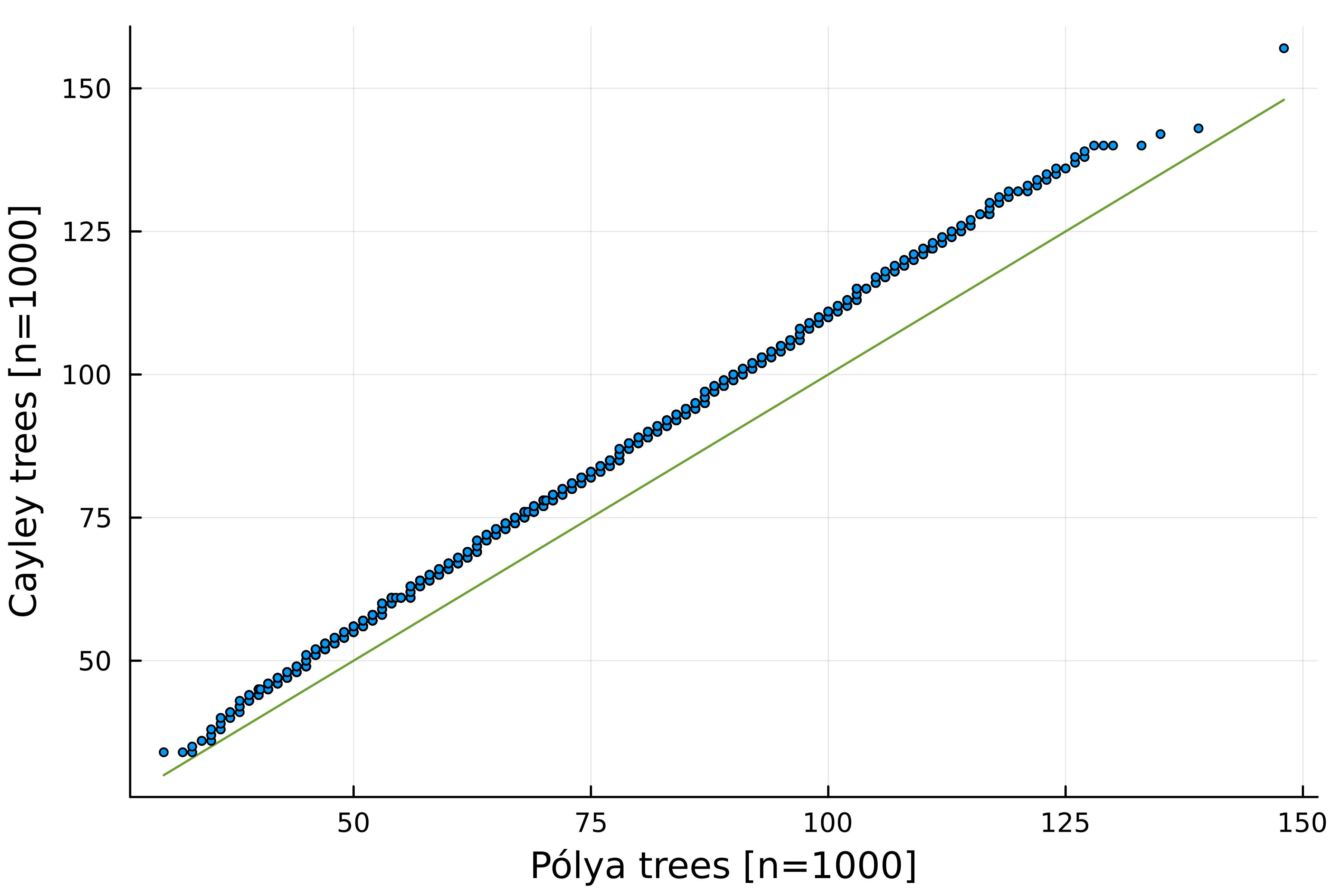}\]

Another continuous function is the path length $I_n$ of a random Pólya tree, namely the sum of the all distances from vertices to the root. According to~\cite[Theorem 4.9]{drmota;trees}, the normalized path length $I_n/n^{3/2}$ converges in distribution to $2/\sigma\int_0^1 E_t dt$; this last term $\omega_+=\int_0^1 E_t dt$ is often called the \emph{Brownian excursion area}, and its distribution, known as the \emph{Airy distribution} has been computed by Takács~\cite[Theorem~5]{takacs;bernoulli-excursion}: recall Kummer's confluent hypergeometric series
\[U(a,b;z)=\frac{\Gamma(1-b)}{\Gamma(a+1-b)} {}_1F_1(a,b;z) + \frac{\Gamma(b-1)}{\Gamma(a)}z^{1-b}{}_1F_1(a+1-b,2-b;z),\]
and let $\alpha_1\approx2.3381,\alpha_2\approx4.0879,\dots$ be the absolute values of the negative zeros of the Airy function $Ai$; then $\omega_+$ is distributed as
\[\frac{2^{13/6}3^{-3/2}}{x^{10/3}}\sum_{i\ge1}\exp(-2\alpha_i^3/27x^2)\alpha_i^2 U(-5/6,4/3;2\alpha_i^3/27x^2).\]
Again we see a good agreement between experimental data and the theoretical distribution, but with a significantly different parameter $\sigma$:
\[\includegraphics[width=10cm]{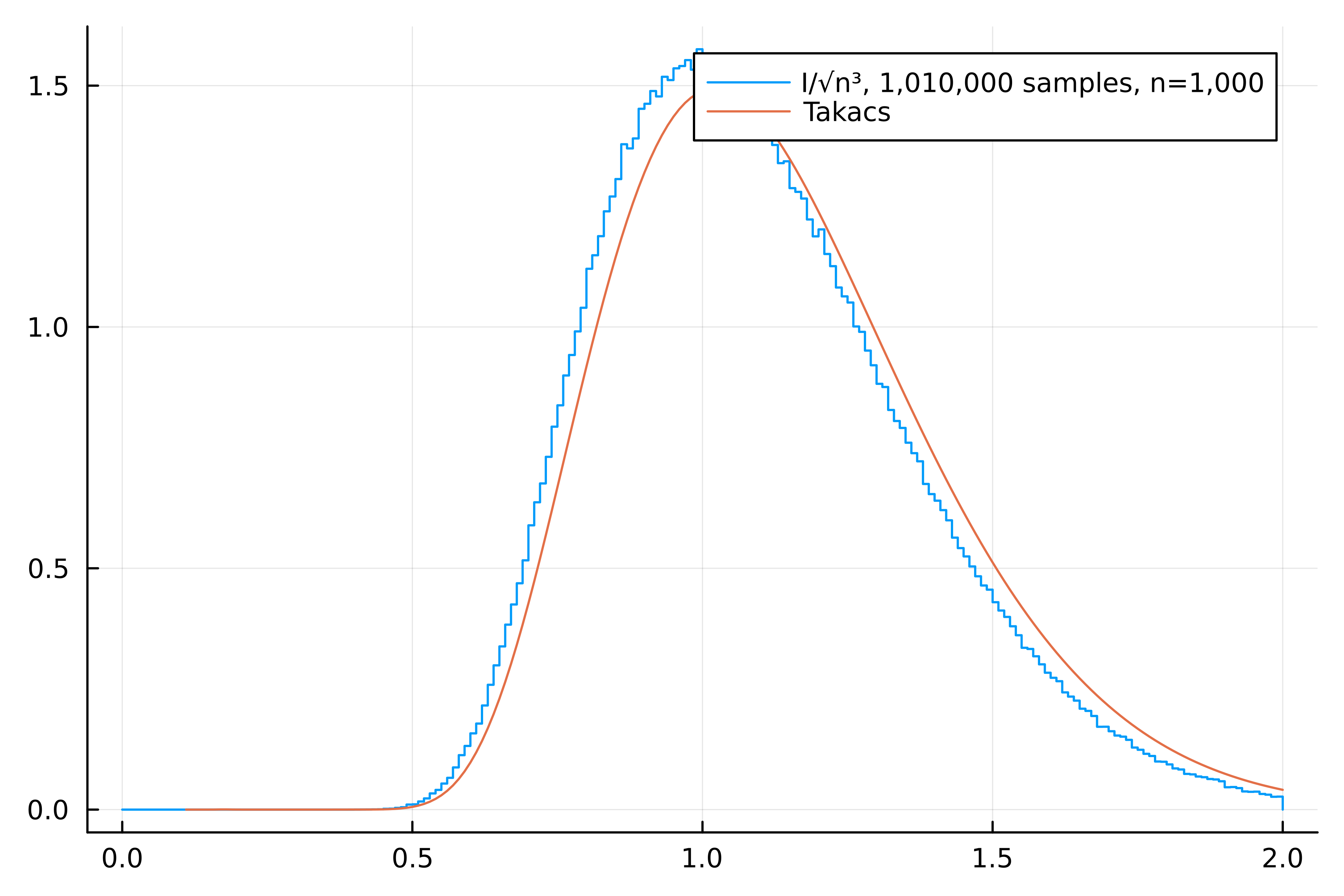}\]

\subsection{Growth}
Recall that the growth function $L_n(k)$ is the number of vertices at distance $k$ from the root in a random $n$-vertex Pólya tree. These parameters are also continuous, and admit limiting distributions based on the Brownian excursion; setting $\ell_n(t)=n^{-1/2} L_n(t n^{1/2})$, we have by~\cite[Theorem 4.10]{drmota;trees} that $\ell_n(t)$ behaves as $\sigma/2 \ell(t\sigma/2)$ for $\ell$ the \emph{occupation measure}
\[\ell(t)=\lim_{\epsilon\to0}\epsilon^{-1}\int_0^1 \mathbf1_{[t,t+\epsilon]}(E_s)ds.\]

In particular, the maximum of $L_n$ is the \emph{width} $W_n$, and since the maximum is a continuous functional we obtain, according to~\cite[Theorem 4.11]{drmota;trees},
\[W(T_n)/\sqrt n\approx 2\frac{\sqrt8}{b\sqrt\rho}\sup_{t\in[0,1]}\ell(t);\]
note then that the distribution of $\sup\ell(t)$ is twice that of $\sup E_t$.

We can compare the widths of 1\,000\,000 random trees with 1000 vertices to the analytic estimate, and fit a function of the form $dM(\tau x)/d(\tau x)$ for optimal $\tau$ to the data; the predicted $\tau$ is $2b\sqrt{\eta/8n}$. The result is
\[\includegraphics[width=10cm]{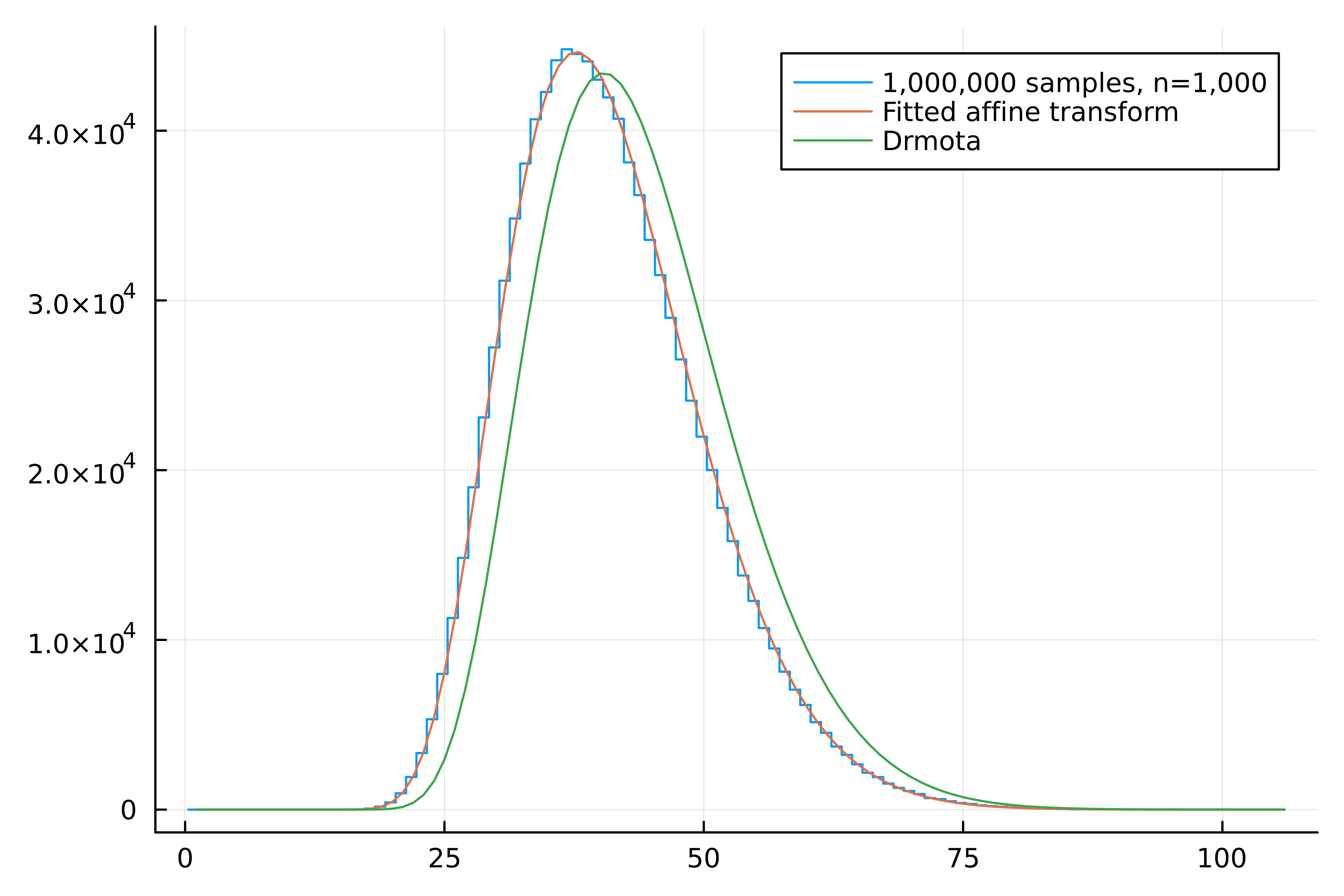}\]

Once more, the optimal, empirical parameter gives an extremely good fit, while the analytic value of the parameter is quite far from the observations.

The comparison with Cayley trees also gives quite clear results; the distributions are similar, but with a different scale. Here is the qq plot of the width of Cayley trees compared to Pólya trees:
\[\includegraphics[width=7cm]{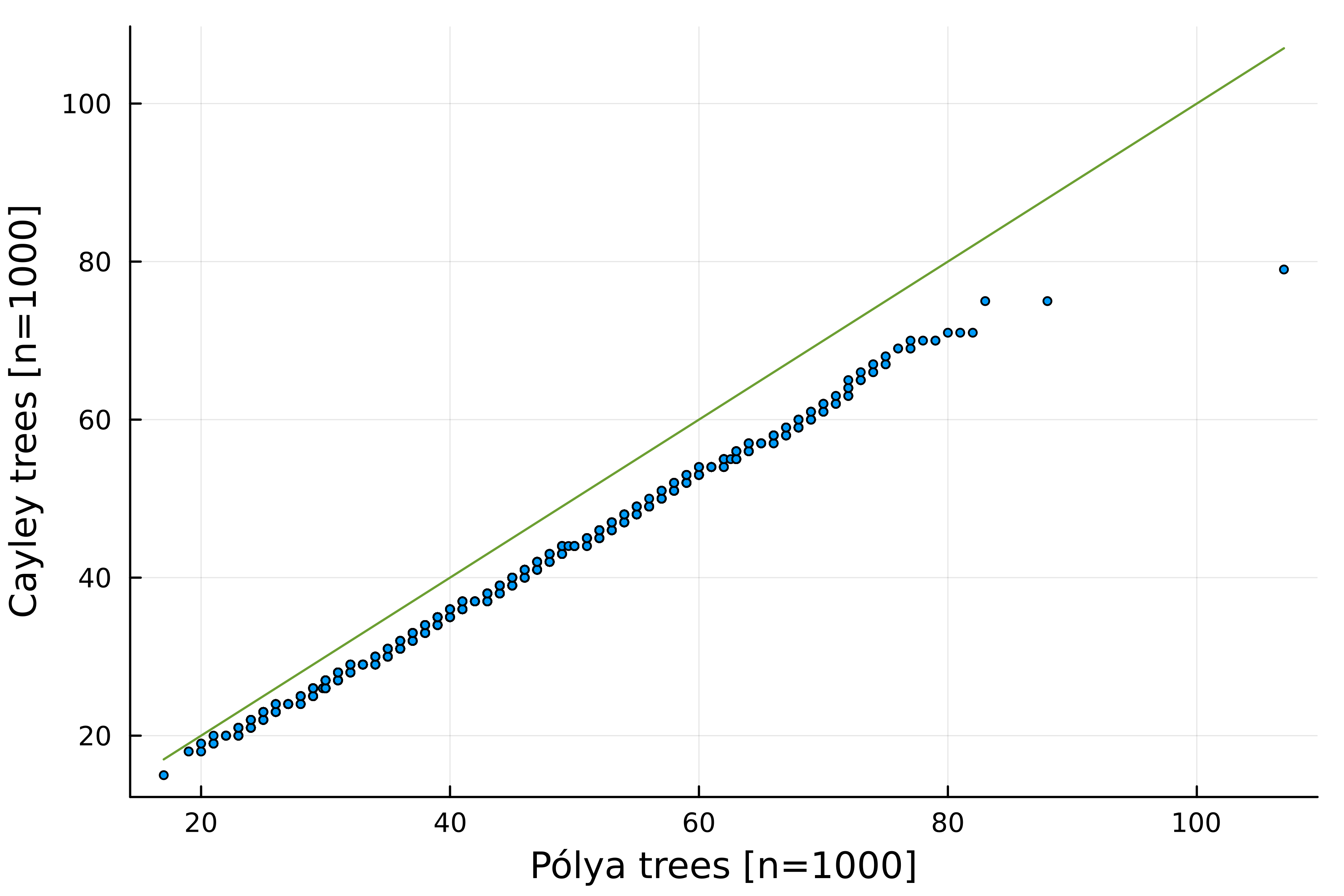}\]

\subsection{Degrees}\label{ss:degrees}
The degree distribution in a random Pólya tree is of a quite different nature, and cannot be approached at all by a limiting model such as the CRT. One of the most practical methods involves generating functions: recall the generating function $t(z)$ for Pólya trees from~\eqref{eq:functional}, and note that there is a similar-looking functional equation for the generating function $f_m$ of Pólya trees with out-degree at most $m$; this already appears in Otter~\cite{otter;trees}.

\begin{figure}[h!]
  \centerline{\includegraphics[width=6cm]{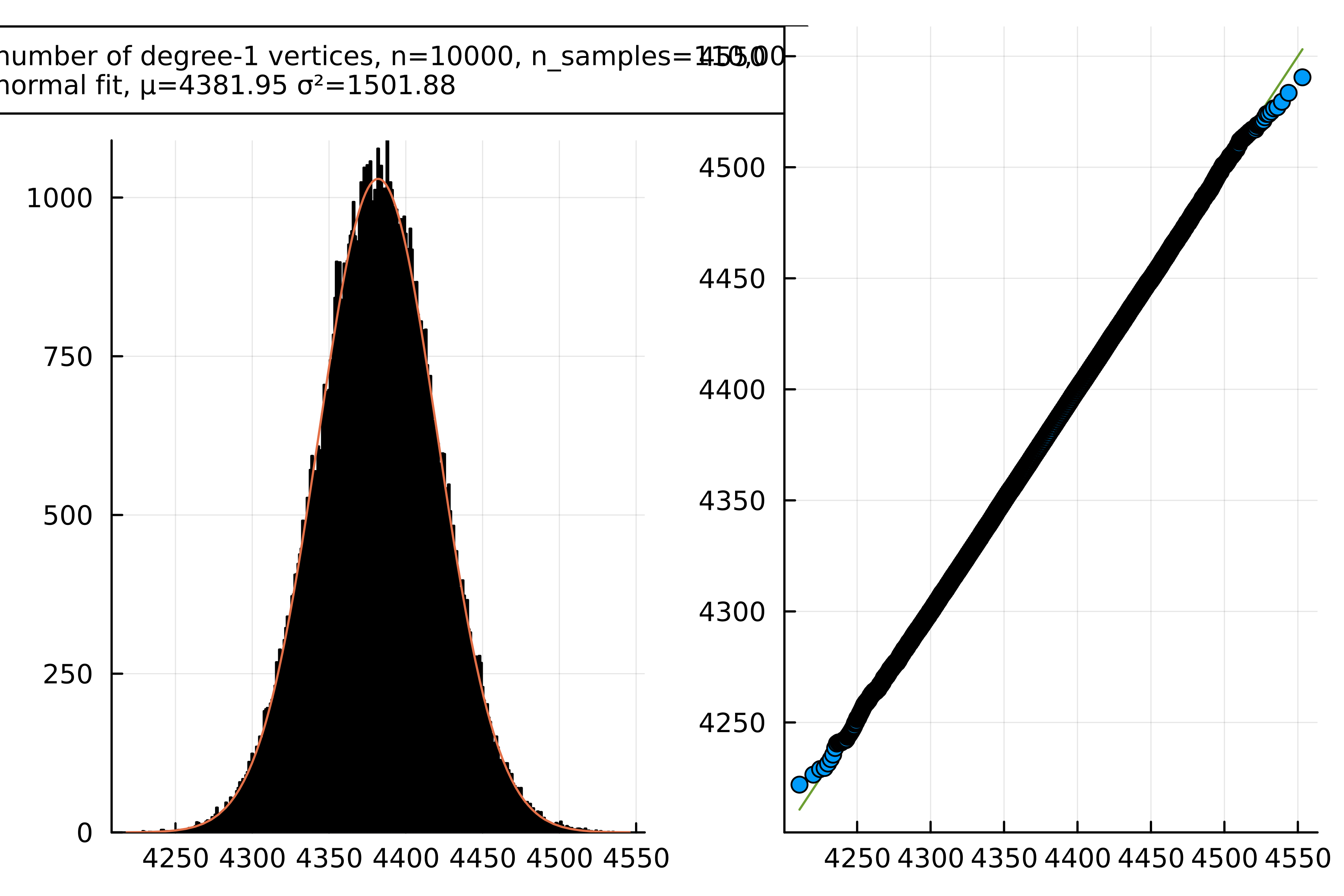}\quad\includegraphics[width=6cm]{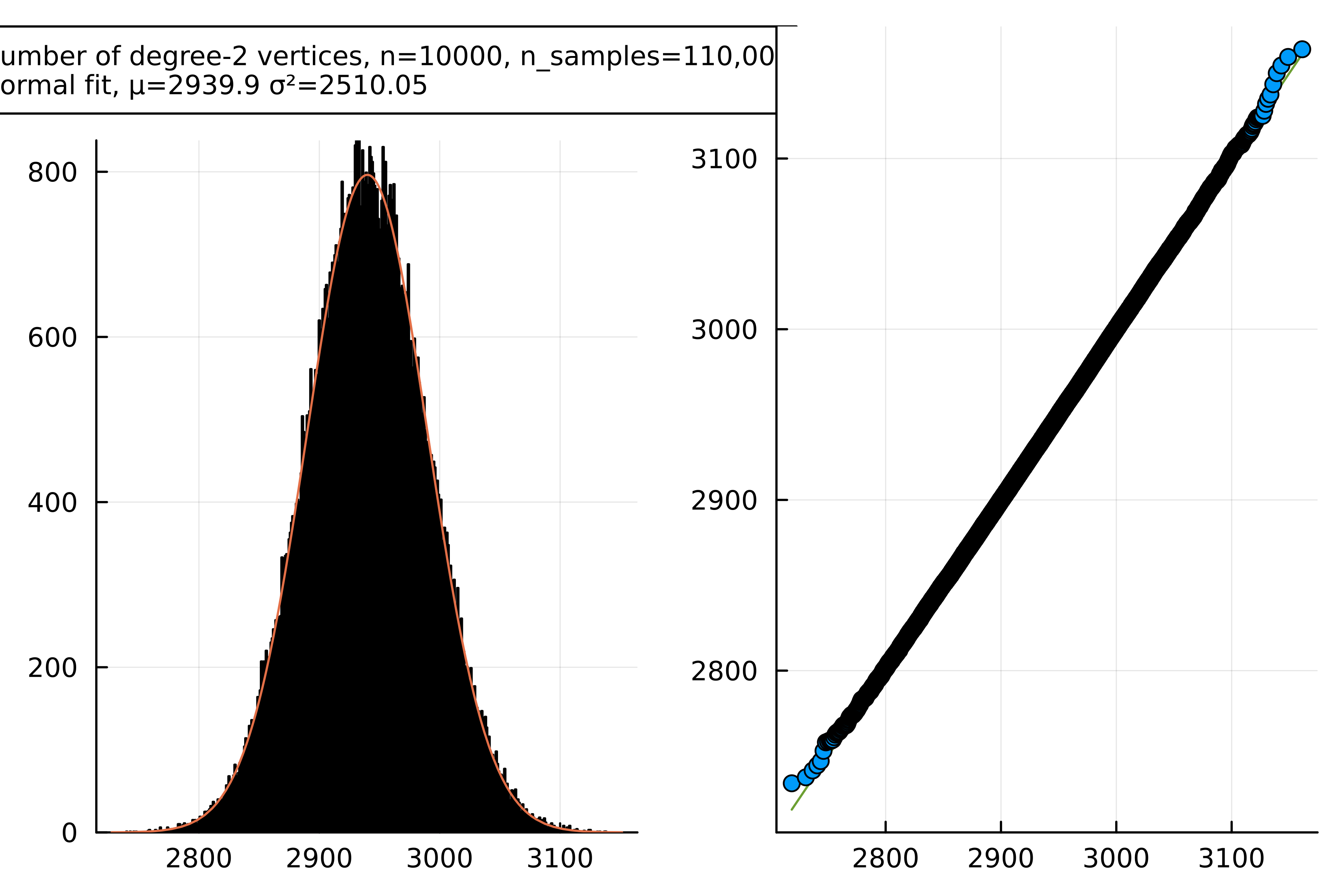}}

  \centerline{\includegraphics[width=6cm]{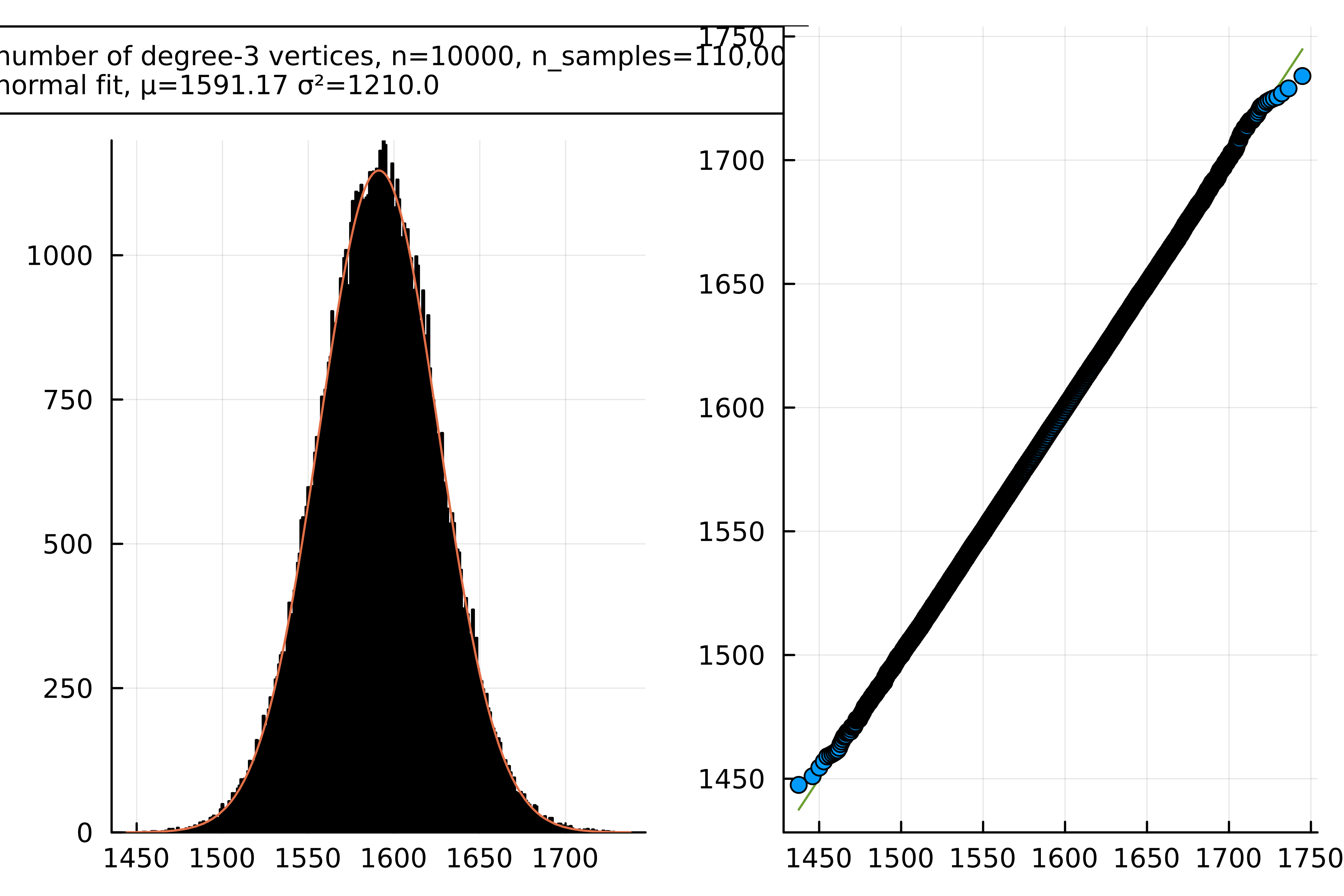}\quad\includegraphics[width=6cm]{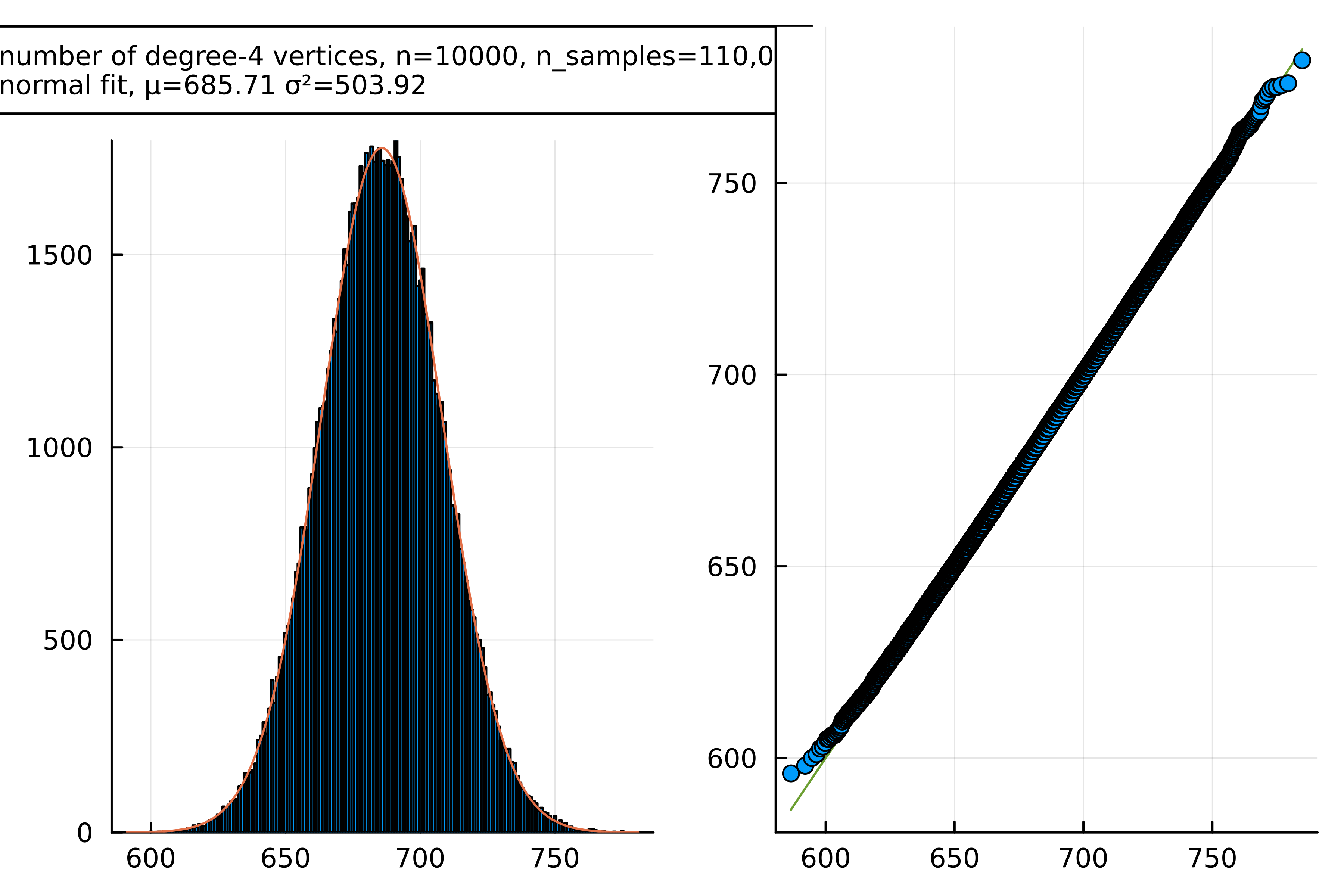}}

  \centerline{\includegraphics[width=6cm]{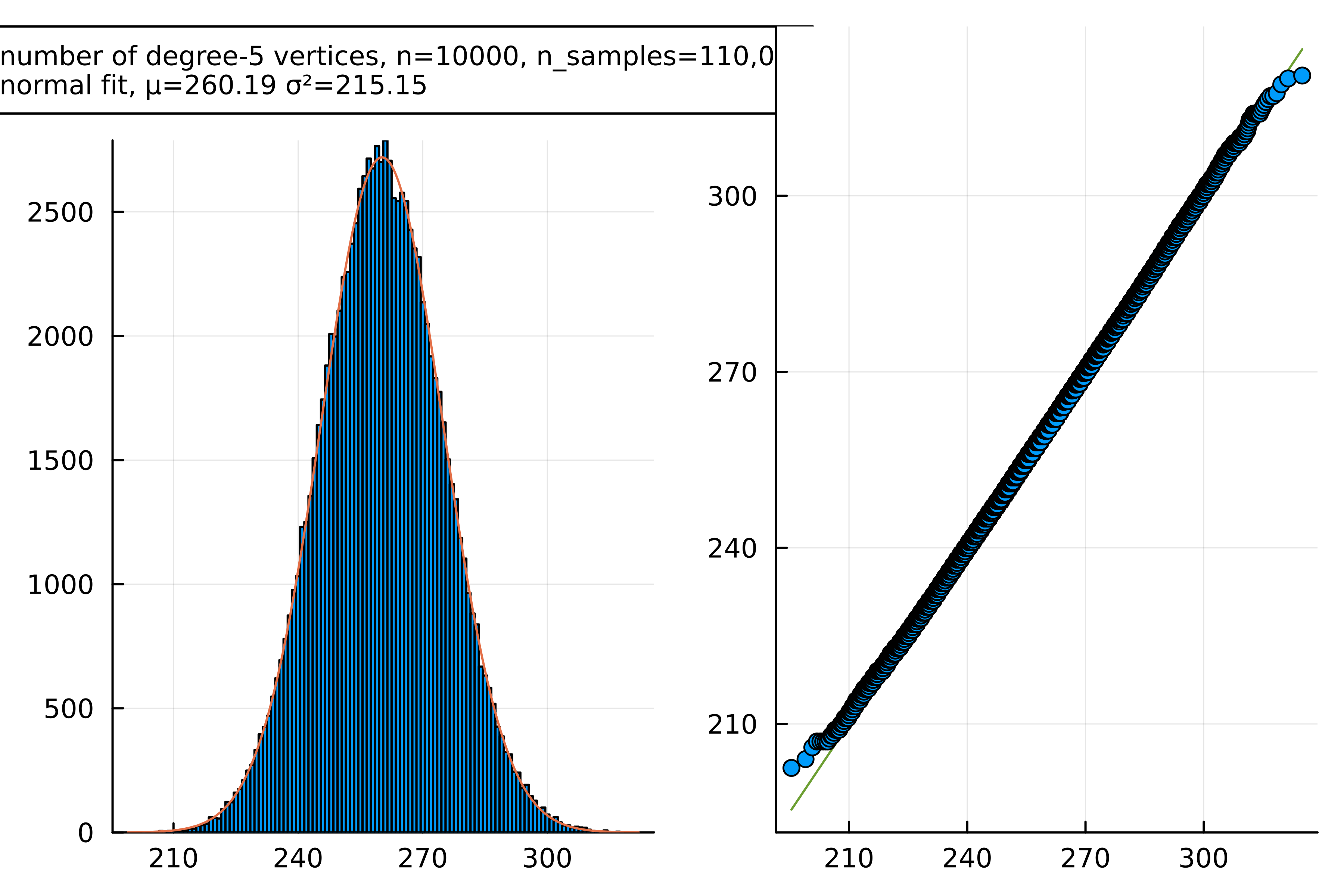}\quad\includegraphics[width=6cm]{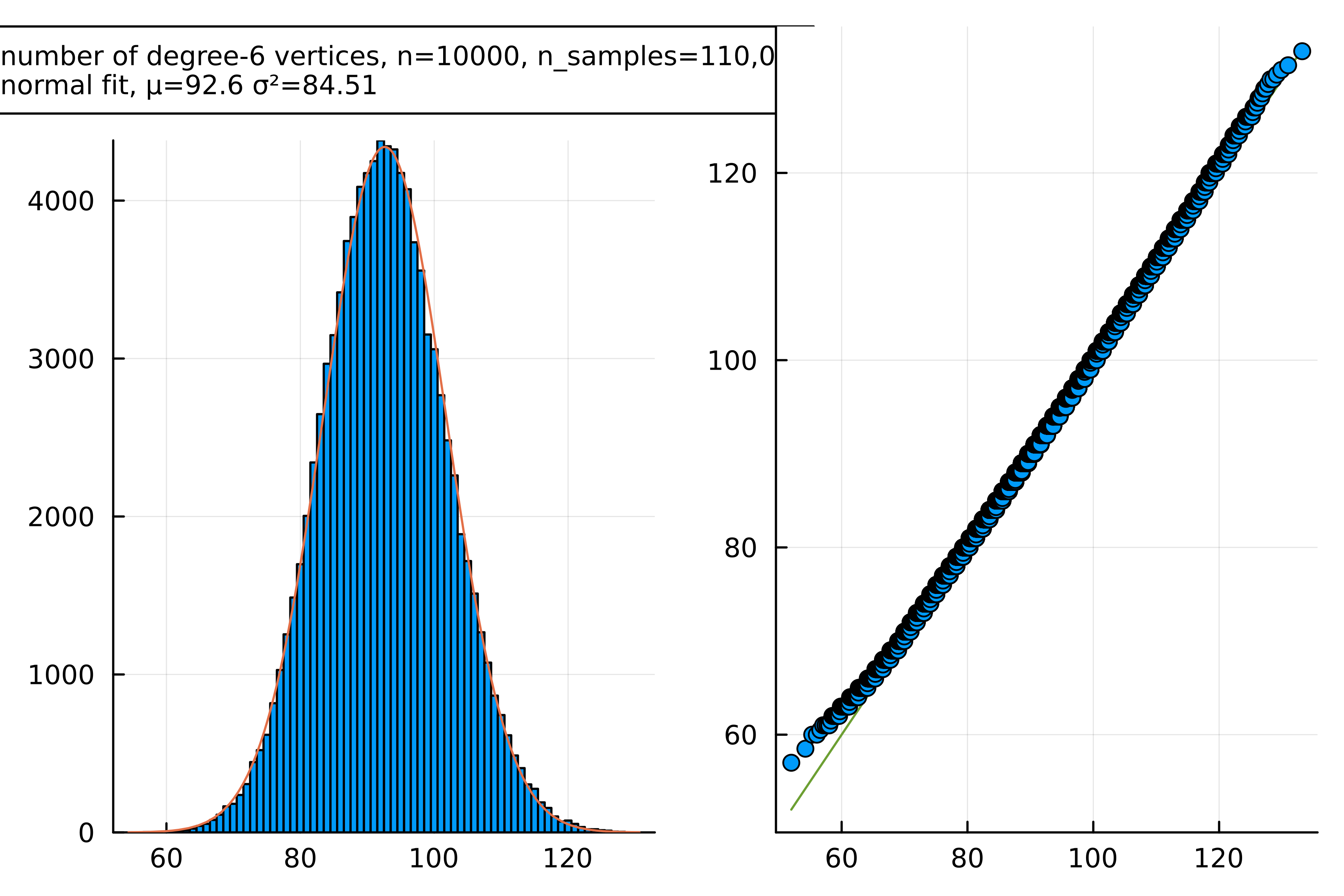}}

  \centerline{\includegraphics[width=6cm]{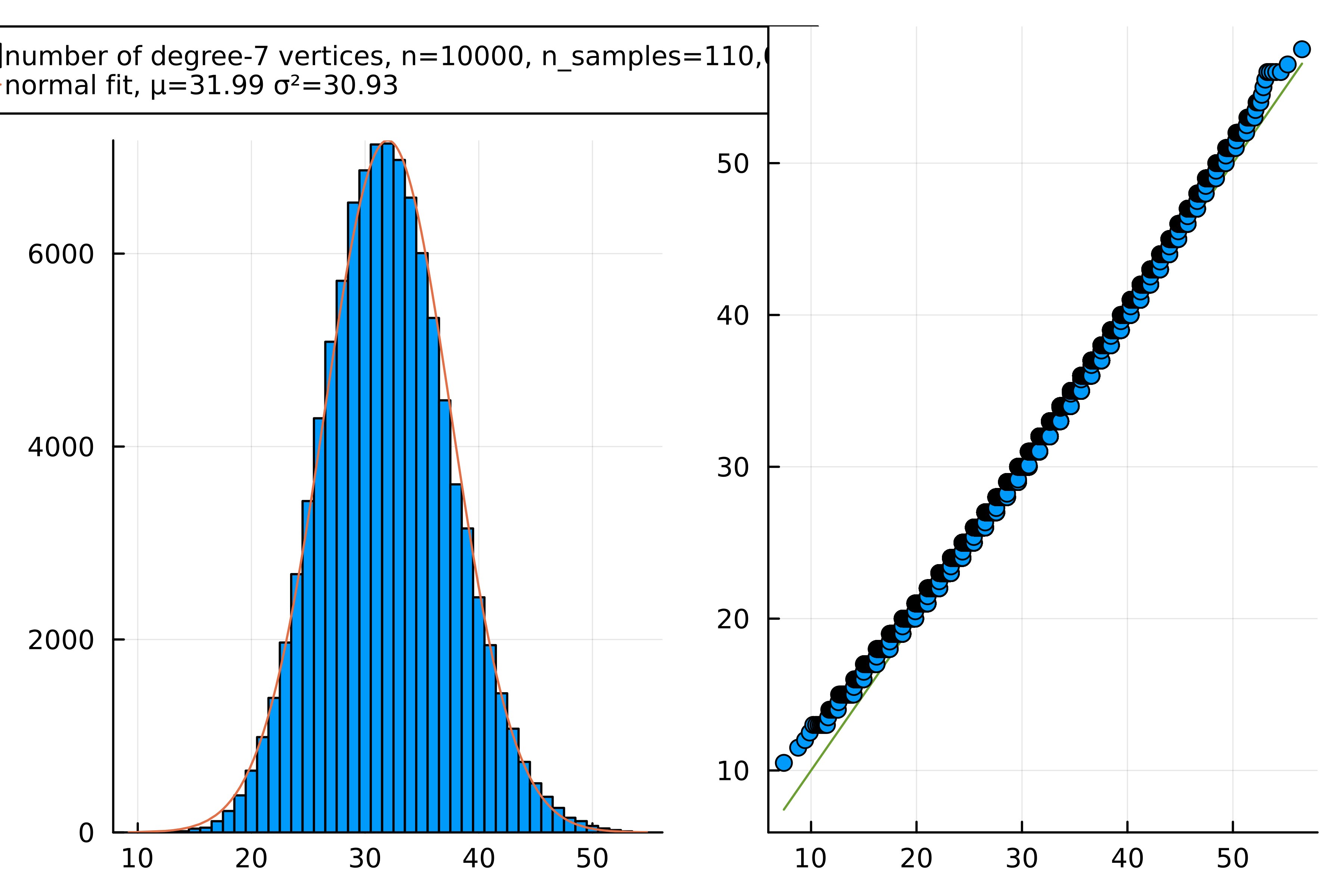}\quad\includegraphics[width=6cm]{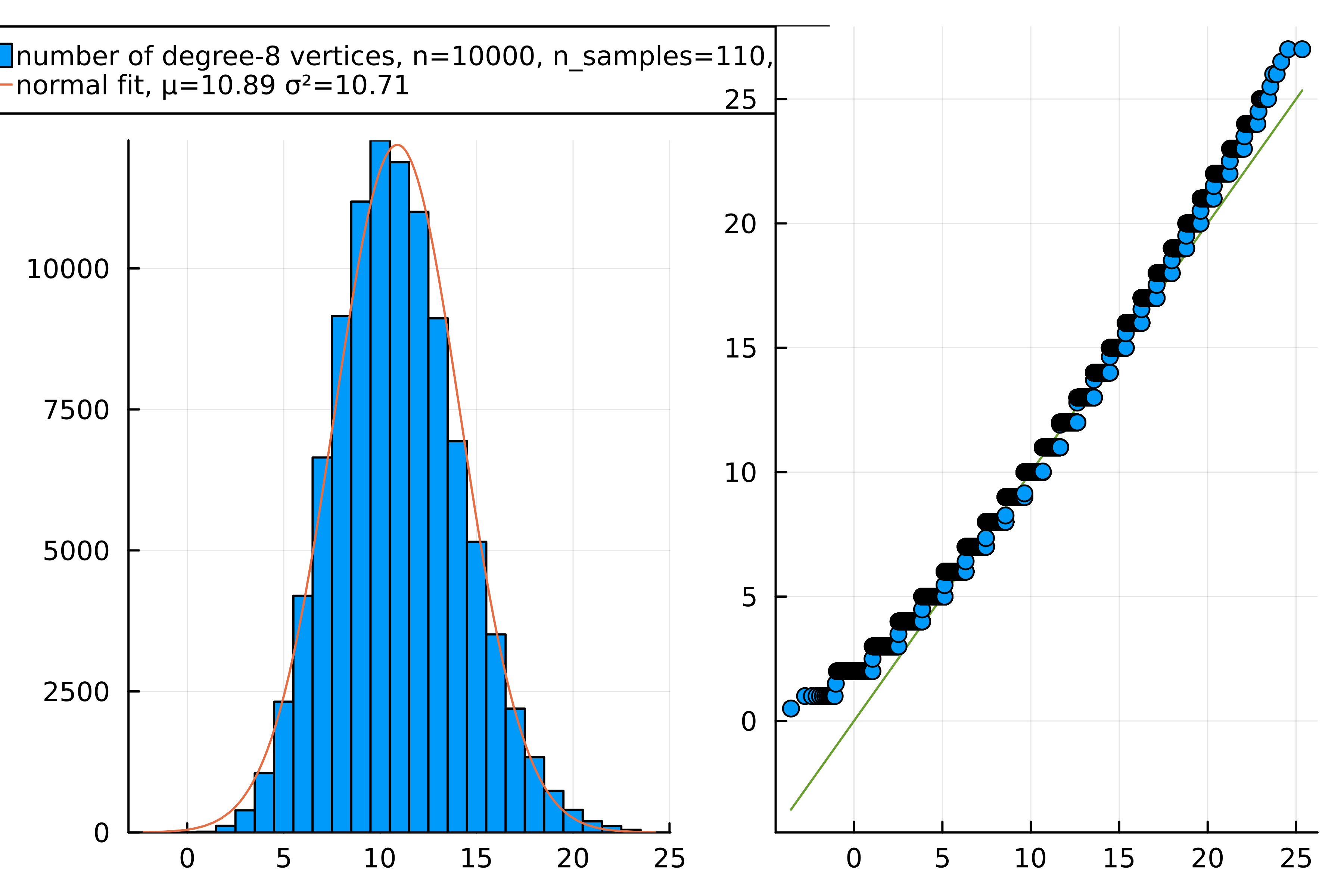}}
  \caption{The distributions of degree-$n$ vertices in a random Pólya tree, for $n=1,\dots,8$}\label{fig:degreedistribution}
\end{figure}

All the functions $f_m$ have a root singularity at $\rho_m$ on their circle of convergence, leading to the asymptotic count $t_{n,m}\approx b_m\rho_m^n n^{-3/2}$ of Pólya trees with maximal out-degree $m$. Here $t_{n,\infty}=t_n$ and $\rho_\infty=\rho$ is the constant we had before.

Goh and Schmutz~\cite{goh-schmutz;unlabeled} show that $\rho_m$ behaves quite precisely as
\[\rho_m \approx \rho + c \rho^{m+1},\]
for some constant $c\approx1.1103$. This immediately leads to
\[\mathbb P(\text{max.deg.}\le m)\approx\frac{b\rho^n n^{-3/2}}{b_m\rho_m^n n^{-3/2}}\approx\exp(-c n\rho^m).\]
This means that the maximum degree has expected value $-\log(n)/\log(\rho)$, doubly-exponential decay below its expectancy, and exponential convergence to $1$ above its expectancy.

Note that this is quite different from the maximum degree of a random labeled tree, whose expectancy is $\propto \log n/\log\log n$ and concentrates on three values, see Theorem~\ref{thm:polyamaxdeg}. For comparison, here is a plot with both distributions, where the circle radius is proportional to the logarithm of the distribution density:
\[\includegraphics[width=7cm]{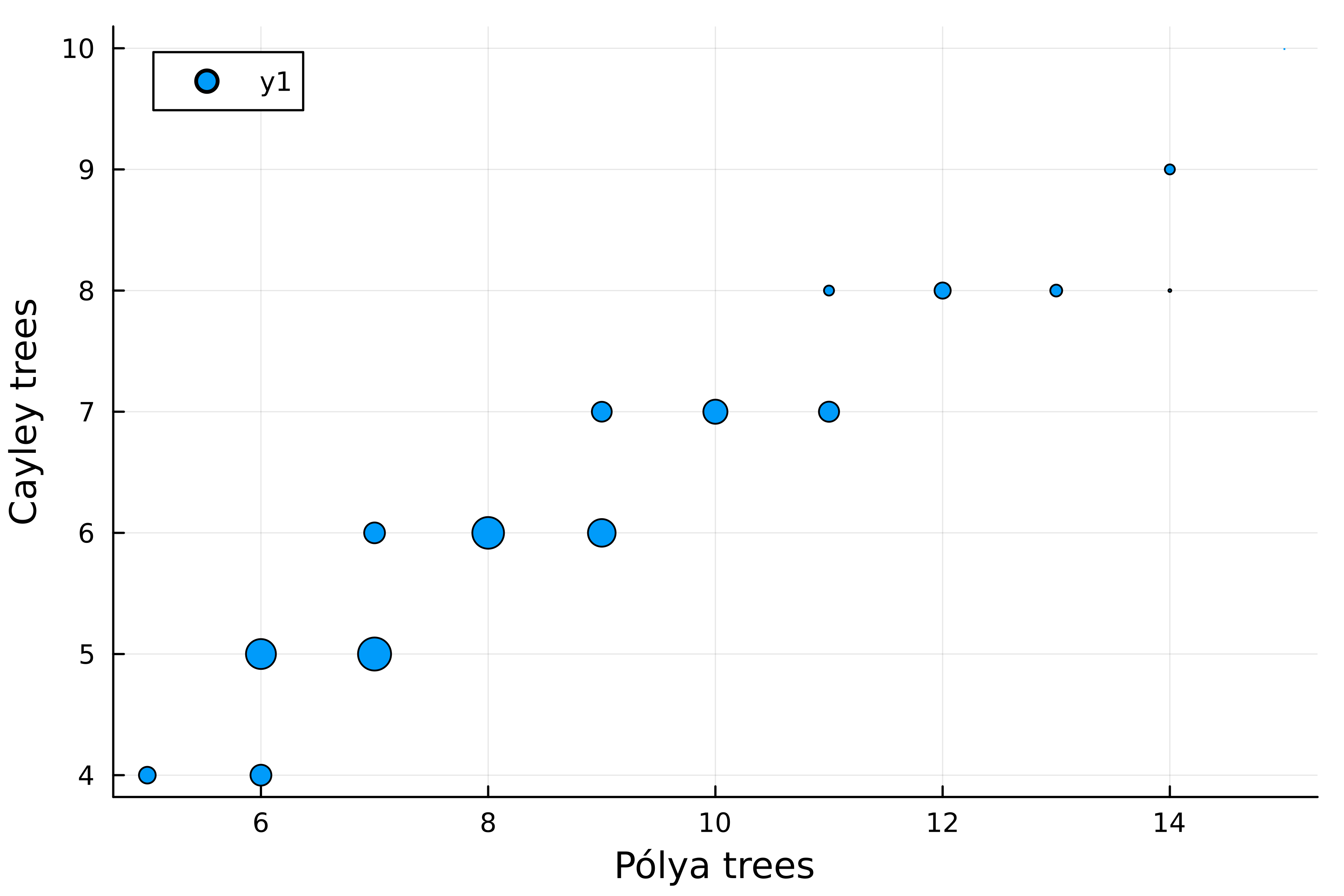}\]

Here the fit between data and predictions is excellent: already for $n=100$ the curves are indistinguishable to the eye, in their cumulative distributions,
\[\includegraphics[width=10cm]{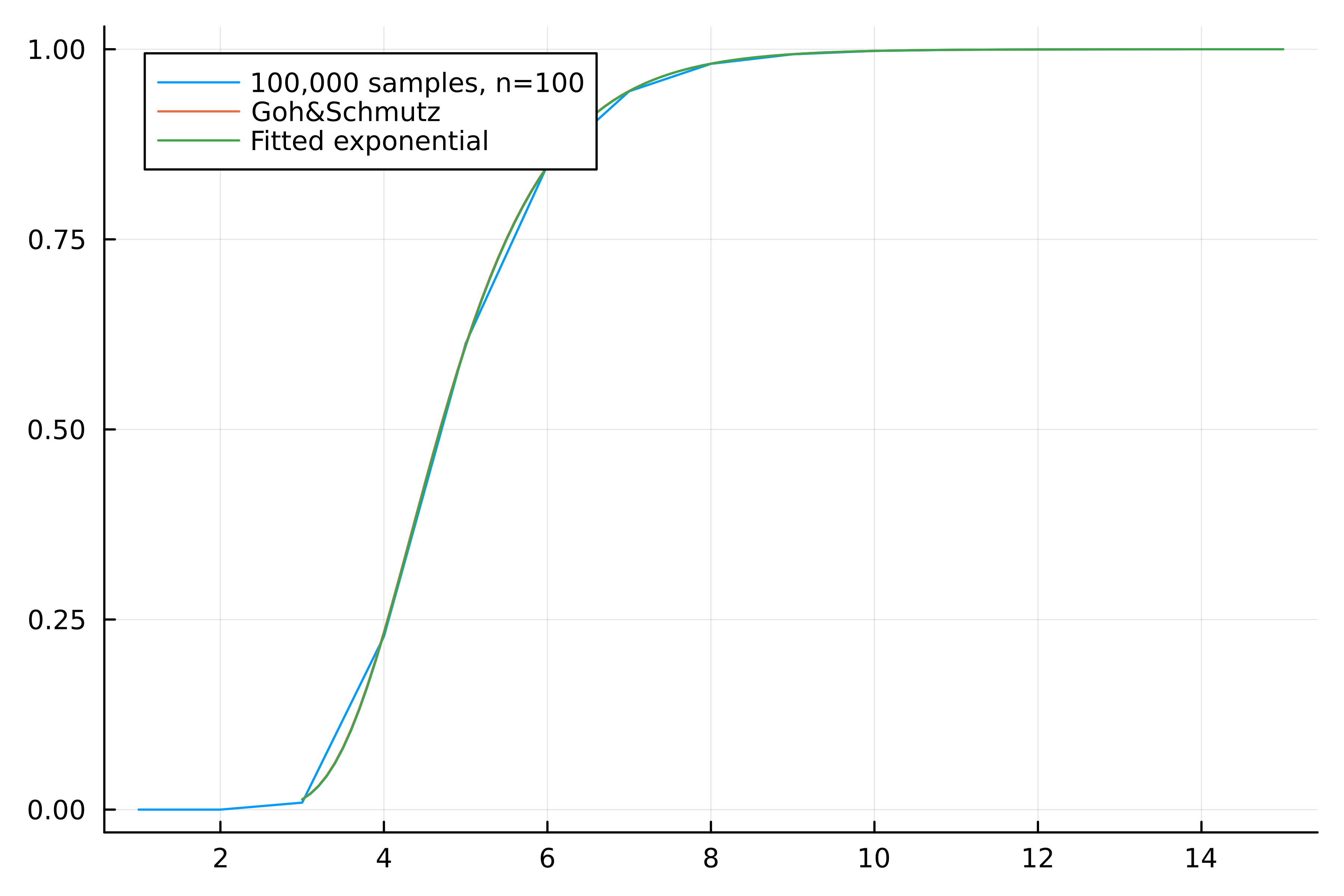}\]
and as a qq-plot comparing to random values sampled according to the cumulative distribution $\exp(-c n\rho^m)$, namely samples $\log(-\log(x)/c)/log(\rho)$ for $x$ drawn uniformly from $[0,1]$:
\[\includegraphics[width=10cm]{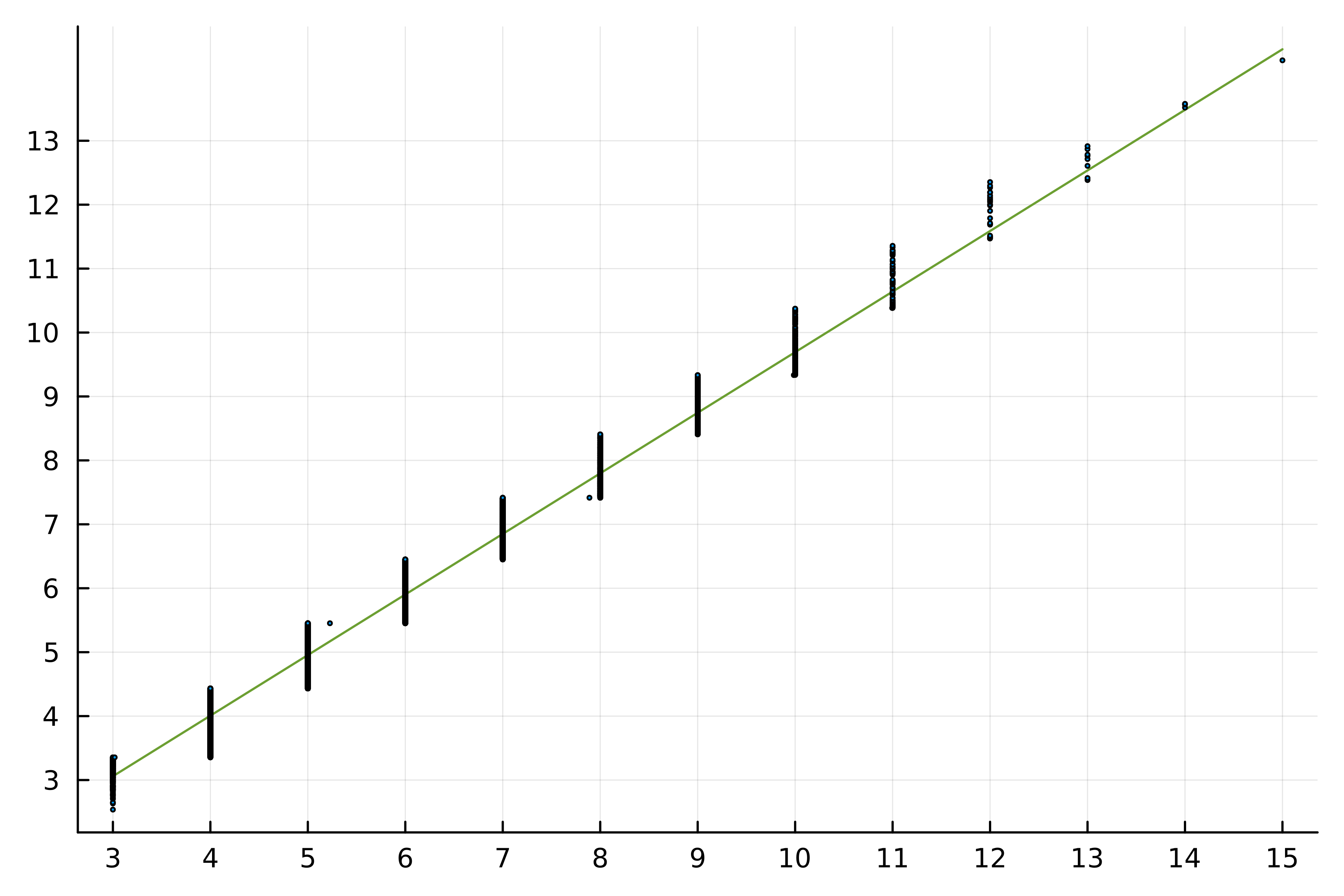}.\]
However, for that last plot, the optimal fitting parameter $c$ was $1.8$, far off from the $1.1103$ predicted by theory.

We finally turn to the number of leaves, or more generally of vertices of given small degree. We increase, here, the number of vertices to $n=10\,000$, and compute $100\,000$ samples, comparing their distribution of degrees (which is normal) to the values predicted in~\cite[Table~2]{robinson-schwenk;distribution}, which we repeat for convenience:
\[\begin{array}{r|cccccccccc}
    \text{degree} & 1 & 2 & 3 & 4 & 5\\
    \mathbf P(\deg) & 0.438\,156 & 0.293\,998 & 0.159\,114 & 0.068\,592 & 0.026\,027\\ \hline
    \text{degree} & 6 & 7 & 8 & 9 & 10\\
    \mathbf P(\deg) & 0.009\,259 & 0.003\,198 & 0.000\,985 & 0.000\,355 & 0.000\,316
  \end{array}\]
The means of the histograms in Table~\ref{fig:degreedistribution} closely match the exact asymptotics for the mean degrees of above table. For example the mean of the degree $1$ vertices is 0.4382, in agreement with the theoretical mean $0.438\,156$. This match of our simulations to theory is one more verification that our code is working as it should.

\subsection{Estimating the constants $b,\rho$}
The constants $b,\rho$ appearing in Theorem~\ref{thm:otter} may be determined to high precision as follows. Consider a large parameter $n$ (in our experiments, $40$ is plenty), and the map
\[F\colon\R^{n+1}\to\R^{n+1},\;((u_1,\dots,u_n),\rho)\mapsto((u_i-\rho^i\exp(\sum_{1\le j\le n/i} u_{i\cdot j}/j)_{i=1,\dots,n},u_1-1).\]
Then $\rho$ is the last coordinate of a zero of $F$, and for $\rho_0=0.338$ the point $((\rho_0^i)_{i=1,\dots,n},\rho_0)$ is in its basin of attraction, so Newton's method converges very fast to a solution $\rho$ with hundreds of digits. Then to compute $b$ we replace the last coordinate of $F$ by $u_1-1+\epsilon$, find a solution $\rho_\epsilon$, and estimate $b$ as $\epsilon/\sqrt{\rho-\rho_\epsilon}$. This merely requires thrice more significant digits from $\rho$ than we get for $b$.


\bibliographystyle{amsalpha}
\bibliography{trees}

\appendix

\section{Code and remarks on performance}

All experiments were performed using the programming language \textsc{Julia}. Its many advantages include speed, parallelism, a large collection of libraries implementing statistics, plotting, and special functions; and the ability to link in external code such as the DejaVu functions.

The experiments are all part of a Jupyter notebook, which can be downloaded as part of the resources. The degree count data in~\S\ref{ss:degrees} has been extracted, and can be downloaded in plain text format (one column per degree, $100\,000$ rows). The data are available at\\
\centerline{\url{https://doi.org/10.5281/zenodo.14186424}}

\section{Conclusions}
We have devised and implemented an efficient, robust procedure for sampling Pólya trees. Using it, we have compared experimental data to the existing literature. Our simulations show a systematic departure from the asymptotics. Perhaps there is a correction term which eventually tends to zero but slowly enough (or with a large enough constant) to be visible for present sample sizes.

We also do not have any theoretical foundation for the performance of our algorithm. The data we obtained are compatible with the possibility that the Markov chain mixes in constant (i.e.\ independent of $n$) time, and that $20$ steps of the Burnside process produce a Pólya tree that is for all practical purposes indistinguishable from a uniform random sample.

\end{document}